\documentclass[10pt, reqno]{amsart}
\usepackage[utf8]{inputenc}  

\usepackage[all]{xy}
\usepackage{amsthm}
\usepackage{amsfonts}
\usepackage{amsmath}
\usepackage{amssymb}

\usepackage{dsfont}
\usepackage{marvosym}
\usepackage{mathpartir}
\usepackage{mathtools} 
\usepackage{stmaryrd}

\usepackage{xcolor}
\definecolor{darkgreen}{rgb}{0,0.45,0}
\definecolor{darkred}{rgb}{0.75,0,0}
\definecolor{darkblue}{rgb}{0,0,0.6}
\usepackage[colorlinks,citecolor=darkgreen,linkcolor=darkred,urlcolor=darkblue]{hyperref}
\usepackage[mathscr]{eucal}
\usepackage{enumerate} 
\usepackage{rotating}
\usepackage[capitalize]{cleveref}

\usepackage{mathpartir}

\usepackage{tikz}
\usetikzlibrary{arrows}
\usepackage[all]{xy}
\xyoption{2cell}
\xyoption{curve}
\UseTwocells
\input{diagxy}  

\newcommand{\adjRelayII}[3][2.2em]{\ensuremath{\SelectTips{cm}{10}\xymatrix@C=#1@1{{#2} \ar@<1ex>[r]^-{\ArgI}^-{}="1" & {#3} \ar@<1ex>[l]^-{\ArgII}^-{}="2" \ar@{}"1";"2"|(.3){\hbox{}}="7" \ar@{}"1";"2"|(.7){\hbox{}}="8" \ar@{|-} "8" ;"7"}}}

\theoremstyle{plain}
\newtheorem{theorem}{Theorem}[section]

\newtheorem{proposition}[theorem]{Proposition}
\newtheorem{lemma}[theorem]{Lemma}
\newtheorem{corollary}[theorem]{Corollary}

\theoremstyle{definition}
\newtheorem{definition}[theorem]{Definition}
\newtheorem{example}[theorem]{Example}
\newtheorem{notation}[theorem]{Notation}
\newtheorem{remark}[theorem]{Remark}

\newcommand{\Gpd}{\mathsf{Gpd}} 

\newcommand{\C}{\mathscr{C}} 
\newcommand{\D}{\mathscr{D}} 

\newcommand{\E}{\mathcal{E}}
\newcommand{\ff}{\mathcal{F}}

\newcommand{\Ll}{\mathcal{L}}
\newcommand{\M}{\mathcal{M}}

\newcommand{\Rr}{\mathcal{R}}
\newcommand{\U}{\mathcal{U}} 
\renewcommand{\phi}{\varphi}

\newcommand{\dom}{\mathrm{dom}} 
\newcommand{\ev}{\mathrm{ev}}  
\newcommand{\name}[1]{{\ulcorner {#1} \urcorner}}
\newcommand{\Ho}{\operatorname{Ho}} 
\newcommand{\id}{\mathrm{id}}

\newcommand{\li}[1]{\overline{#1}} 
\newcommand{\ti}[1]{\widetilde{#1}} 
\newbox\anglebox
\setbox\anglebox=\hbox{\xy \POS(75,0)\ar@{-} (0,0) \ar@{-} (75,75)\endxy}
\def\angle{\copy\anglebox} 

\newcommand{\intperp}{{\begin{sideways}$\!\Vdash$\end{sideways}}} 
\newcommand{\iperp}{\ \intperp\ } 

\DeclareMathOperator{\Cart}{Cart}
\DeclareMathOperator{\hofib}{hofib}
\DeclareMathOperator{\Eq}{Eq}

\DeclareMathOperator{\idtoequiv}{\mathsf{idtoequiv}}

\DeclareMathOperator{\IsLocal}{\mathsf{IsLocal}}

\DeclareMathOperator{\Path}{Path}
\DeclareMathOperator{\pr}{pr}

\DeclarePairedDelimiterX{\norm}[1]{\lVert}{\rVert}{#1}
\DeclarePairedDelimiterX{\abs}[1]{\lvert}{\rvert}{#1}

\makeatletter
\newbox\xrat@below
\newbox\xrat@above
\newcommand{\xrightarrowtail}[2][]{%
  \setbox\xrat@below=\hbox{\ensuremath{\scriptstyle #1}}%
  \setbox\xrat@above=\hbox{\ensuremath{\scriptstyle #2}}%
  \pgfmathsetlengthmacro{\xrat@len}{max(\wd\xrat@below,\wd\xrat@above)+.6em}%
  \mathrel{\tikz [>->,baseline=-.75ex]
                 \draw (0,0) -- node[below=-2pt] {\box\xrat@below}
                                node[above=-2pt] {\box\xrat@above}
                       (\xrat@len,0) ;}}
\makeatother

\begin{document}
\title{Localization Theory in an $\infty$-topos}

\author{Marco Vergura}
\email{mvergura@uwo.ca}
\date{8 July 2019}
\subjclass[2010]{Primary 55P60;
Secondary  18E35}

\begin{abstract}
Inspired by \cite{locinhott}, we develop the theory of \emph{reflective subfibrations} on an $\infty$-topos $\E$. A reflective subfibration $L_\bullet$ on $\E$ is a pullback-compatible assignment of a reflective subcategory $\D_X\subseteq \E_{/X}$, for every $X \in \E$. Reflective subfibrations abound in homotopy theory, albeit often disguised, e.g., as stable factorization systems. We prove that $L$-local maps (i.e., those maps that belong to some $\D_X$) admit a classifying map, and we introduce $L$-\emph{separated} maps, that is, those maps with $L$-local diagonal. $L$-separated maps are the local class of maps for a reflective subfibration $L'_\bullet$ on $\E$. We prove this fact in the companion paper \cite{l'loc}. In this paper, we investigate some interactions between $L_\bullet$ and $L'_\bullet$ and explain when the two reflective subfibrations coincide.
\end{abstract}
\maketitle
\tableofcontents

\section{Introduction}
This is the first of a series of two papers on localizations in an $\infty$-topos $\E$, studied via \emph{reflective subfibrations}. The companion paper is \cite{l'loc}.\par 

In classical homotopy theory, localization of spaces gives tools and techniques analogous to those provided in algebra by localizations of rings and modules. For example, localization of spaces at primes allows one to \emph{simplify} some problems by working locally at each prime, and it introduces \emph{local-to-global principles} and \emph{fracture theorems} in homotopy theory (see \cite{moreconcise} for an overview). Localization can also be seen as a way to \emph{present} certain objects of interest, rather than as a tool for simplifying them. For example, given a site $\C$, we are interested in the study of sheaves, rather than presheaves, over $\C$, and the process of inverting the covering sieves in the category of presheaves is better understood as a presentation of the category of sheaves, rather than as a simplification of the category of presheaves.\par

It is thus interesting to develop a unifying framework for the study of localization theories, with the goal of encompassing and generalizing some of the examples that are already understood, while also providing new insights about them. Our approach, based on the theory of \emph{reflective subfibrations} on an $\infty$-topos $\E$, employs the language of $\infty$\emph{-category theory}, it emphasizes \emph{localization of maps} rather than localization of objects, and it naturally shares ties with \emph{homotopy type theory}, which has models in $\infty$-topoi. In this paper, we carry out a systematic study of reflective subfibrations $L_\bullet$ on $\E$, show how they encompass many classical examples of localizations, and explore some relationships between $L_\bullet$ and another reflective subfibration $L'_\bullet$ associated to it.  

\subsection{Content and Structure}
In \cref{sec:univforlocclassofmaps}, we follow \cite{htt} and \cite{univinloccarclosed} to give an overview of the theory of local classes of maps, and of univalent classifying maps in an $\infty$-topos $\E$. In order to formulate univalence, one needs to introduce an \emph{object of equivalences} between any two objects $X$ and $Y$ of $\E$, defined in \cite[Thm.~2.10]{univinloccarclosed}. We give an alternative characterization of it in \cref{lm:charofeqxy}.\par

In \cref{sec:univfromreflsubf}, we take from \cite{modinhott} the definition of \emph{reflective subfibrations} on an $\infty$-topos $\E$ and explore some of their properties. A reflective subfibration $L_\bullet$ on $\E$ is a pullback-compatible assignment of reflective subcategories $\D_{X}\subseteq \E_{/X}$, with induced localization functor $L_X$, for every $X\in\E$. The collection of all objects in $\D_X$, as $X$ varies in $\E$, form the $L$\emph{-local maps}, and the objects of $\D:=\D_1$ are called $L$\emph{-local objects}. We show that $L$-local maps form a local class of maps in $\E$, and we characterize reflective subfibrations on spaces as ``fiberwise localizations" (\cref{cor:reflsubfinspaces}). Thus, $L$-local maps admit a univalent classifying map. This classifying map links reflective subfibration on $\E$ to reflective subuniverses in HoTT.\par  

In \cref{sec:lconnmaps}, we introduce and study $L$\emph{-connected maps} (\cref{def:lconnmaps}). Our main result in this section is the following theorem, which is proven, in a different flavor and in HoTT, in \cite{modinhott}.
\newtheorem*{thm:stablefactsystandmod}{\cref{thm:stablefactsystandmod}}
\begin{thm:stablefactsystandmod}
Let $\E$ be an $\infty$-topos.
\begin{enumerate}
\item Let $\ff=(\Ll,\Rr)$ be a stable factorization system on $\E$. There exists a modality $L^\ff_\bullet$ on $\E$ whose local maps are exactly the maps in $\Rr$.
\item Let $L_\bullet$ be a modality on $\E$. Let $\Ll$ be the class of $L$-connected maps and $\Rr$ the class of $L$-local maps. Then $\ff_L=(\Ll,\Rr)$ is a stable factorization system on $\E$.
\end{enumerate}
Moreover, the assignments $\ff\mapsto L^\ff_\bullet$ and $L_\bullet\mapsto\ff_L$ are inverse to one another.
\end{thm:stablefactsystandmod}
\par
In \cref{sec:localizwrtmap}, given any \emph{set} $S$ of maps in $\E$, we prove the existence of $S$\emph{-localization} on $\E$ (\cref{prop:existofsloc}). If $S=\{f\}$, where $f$ is a map of spaces, this recovers localization of spaces at the $f$\emph{-equivalences}, and localization of maps of spaces over a fixed base $X$ at the \emph{fiberwise f-equivalences}.\par

In \cref{sec:lsepmaps}, we introduce $L$\emph{-separated maps} as those maps whose diagonals are $L$-local maps. $L$-separated maps inherit a lot of pleasant properties from $L$-local maps. In particular, they also form a local class of maps of $\E$ (\cref{prop:lsepareloc}), hence they satisfy a necessary condition to be themselves the local maps for a reflective subfibration on $\E$. Showing that this is indeed the case is the purpose of our companion  paper \cite{l'loc}.\par

In \cref{sec:conseqofexistofl'}, we present some results that we can state or prove once we know that $L$-separated maps are associated to a reflective subfibration $L'_\bullet$. For example, we give a way to produce new stable factorization systems from old ones (\cref{prop:l'forstablefactsyst}), and we prove a result that shows how $L'_\bullet$ accounts for the lack of commutativity between $L_\bullet$ and loop functors (\cref{cor:locandloop}). We also give an explicit description of $L'_\bullet$ when $L_\bullet$ is localization at a set $S$ of maps in $\E$ (\cref{prop:l'forfloc}).  We make use of the fact that $L'$-localization is almost left exact (\cref{prop:spantol'locspanislequiv}) to give a characterization of \emph{self-separated} reflective subfibrations. These are reflective subfibrations $L_\bullet$ for which $L_\bullet = L'_\bullet$, and they can be characterized in terms of special localizations of $\E$, the \emph{quasi-cotopological} localizations (\cref{def:cotoploc}), as explained by the following result, which does not appear in \cite{locinhott}.
\newtheorem*{thm:l=l'char}{\cref{thm:l=l'char}}
\begin{thm:l=l'char}
The following are equivalent, for a reflective subfibration $L_\bullet$ of $\E$.
\begin{enumerate}
\item $L_\bullet$ is self-separated.
\item $L_\bullet$ is the modality associated to a quasi-cotopological localization of $\E$.
\end{enumerate}
In this case, hypercomplete maps are $L$-local.
\end{thm:l=l'char}

\subsection{Relation to other work}
Our take on localization theory is inspired by \emph{homotopy type theory}, a dependent type theory with homotopy-theoretic features (\cite{hott}). More precisely, the definition of reflective subfibration on an $\infty$-topos $\E$ appears in \cite{modinhott} as the external notion (in higher topos theory, HTT) that captures the internal description (in homotopy type theory, HoTT) of a \emph{reflective subuniverse}. The work in \cite{modinhott} is mainly concerned $\Sigma$\emph{-closed} reflective subuniverses, also known as \emph{modalities}. These correspond to reflective subfibrations $L_\bullet$ for which the composite of two $L$-local maps is again an $L$-local map. In the context of higher topos theory, the authors of \cite{agenBM} use the term ``modality" as a synonym for a stable factorization system on an $\infty$-topos $\E$, and they carry out a systematic study of these factorization systems. On the other hand, the authors of \cite{locinhott} shift their focus back to the general setting of reflective subuniverses, motivated by the study in homotopy type theory of localizations at primes (which are not modalities). We can then depict our work as the (homotopy) pushout square
$$
\bfig
\node rss(0,0)[\{\text{modalities in HoTT,\ \cite{modinhott}}\}]
\node cors(2000,0)[\{\text{localizations in HoTT,\ \cite{locinhott}}\}]
\node abjm(0,-300)[\{\text{modalities in HTT,\ \cite{agenBM}}\}]
\node th(2000,-300)[\{\text{localizations in HTT,\ this work}\}]

\arrow[rss`cors;]
\arrow[rss`abjm;]
\arrow[cors`th;]
\arrow[abjm`th;]

\place(1950,-110)[\begin{rotate}{180}\angle\end{rotate}]
\efig
$$
In \cite{intlang}, Shulman gives a proof of the conjecture that every $\infty$-topos models homotopy type theory. Hence, all statements proven in HoTT can be translated into true statements in any $\infty$-topos $\E$\footnote{ Modulo the initiality conjecture for homotopy type theory.}. This applies, in particular, to the results in \cite{locinhott}.  However, this recent development does not invalidate our work, for several reasons. First of all, many results here are not present in \cite{locinhott}. Secondly, our work on localization in HTT can not be \emph{immediately} recovered from the analogous work in HoTT since the starting points are different (reflective subfibrations in HTT, reflective subuniverses in HoTT). Furthermore, not all proofs we give here are \emph{direct} translations of the HoTT ones, as some type-theoretic arguments do not have an obvious counterpart in the HTT setting. Even for those arguments that parallel more closely the ones in \cite{locinhott}, our proofs can give some working-knowledge on how to use and adapt HoTT reasoning to prove theorems in an $\infty$-topos $\E$, in a spirit similar to \cite{agenBM} and \cite{rezkprookofbm}.

\subsection*{Acknowledgements} We would like to thank Dan Christensen, for his support and guidance, and Mike Shulman, for the careful reading of the material present here, and for many helpful suggestions.

\subsection*{Notation and Conventions} We will make extensive use of the \emph{orthogonality relation} between maps in an $\infty$-category $\C$. We refer the reader to \cite[Def.~3.1.1]{agenBM} for a definition. Given maps $f$ and $g$ in $\C$, we write $f\perp g$ to mean that $f$ is left orthogonal to $g$, and that $g$ is right orthogonal to $f$.\par
When $\C$ has a terminal object $1$, given an object $X$ in $\C$, we simply write $f\perp X$ to mean $f\perp (X\rightarrow 1)$. To minimize the risk of confusion, we use the symbol $\perp_X$ when we want to denote the orthogonality relation in the slice $\infty-$category $\C_{/X}$. For example, if $\alpha\colon p\rightarrow q$ is a map in $\C_{/X}$ and $r$ is an object in $\C_{/X}$, $\alpha\perp_{X} r$ means that $\alpha$ is left orthogonal to the map $r\rightarrow \id_X$ in $\C_{/X}$.\par
Given a class $\M$ of maps in $\C$, we write $\prescript{\perp}{}{\M}$ for the class of maps in $\C$ that are left orthogonal to every map in $\M$, and we write $\prescript{}{}{\M}^{\perp}$ for the class of maps in $\C$ that are right orthogonal to every map in $\M$.\par
By an $\infty$\emph{-topos} we mean an $\infty$-category $\E$ with the following properties.
\begin{itemize}
\item[(a)] $\E$ is a locally presentable $\infty$-category (\cite[Def.~5.5.0.1]{htt}).
\item[(b)] Colimits in $\E$ are universal (\cite[Def.~6.1.1.2]{htt}).
\item[(c)] The class of all maps in $\E$ is a local class of maps (see \cite[Thm.~6.1.3.8]{htt}).
\end{itemize}
This characterization of $\infty$-topoi follows from \cite[Thm.~6.1.6.8]{htt} and is also given in \cite[Def.~2.2.3]{agenBM}.

\section{Univalence for local classes of maps}
\label{sec:univforlocclassofmaps}
In this section, we develop some background on the theory of local classes of maps and their classifying maps, building on \cite{htt} and \cite{univinloccarclosed}. These notions will allow us to link the theory of reflective subfibrations in higher topos theory to the theory of reflective universes in homotopy type theory.
\subsection{Local classes and classifying maps}
\label{sec:localclassandclassmaps}
We introduce here local classes of maps in an $\infty$-topos $\E$. Modulo size issues, these are the classes $S$ of maps in $\E$ that admit a classifying map. We fix an $\infty$-topos $\E$ throughout.\par

\begin{definition}[{\cite[Def.~6.1.3.8, Prop.~6.2.3.14]{htt}}]
\label{def:locclassofmaps}
A class $S$ of maps in $\E$ is called \emph{local} if it closed under small coproducts in $\E^{\bullet\rightarrow\bullet}$, it is stable under pullbacks and satisfies the following condition. Given any pullback square in $\E$
$$
\bfig
\node e(0,0)[E]
\node m(400,0)[M]
\node x(0,-250)[X]
\node y(400,-250)[Y]

\arrow|a|[e`m;g]
\arrow|l|[e`x;q]
\arrow|r|[m`y;p]
\arrow|b|[x`y;f]

\place(70,-90)[\angle]
\efig
$$ 
where $f$ is an effective epimorphism, $p$ is in $S$ if and only if $q$ is in $S$.
\end{definition}
\begin{definition}
\label{def:classifyingmap}
Let $S$ be a pullback-stable class of maps in an $\infty$-topos $\E$. Let $\Cart(S)$ be the sub-$\infty$-category of $\E^{\bullet\rightarrow\bullet}$ having the maps in $S$ as objects and pullback squares as morphisms. A \emph{classifying map} for $S$ is a terminal object of $\Cart(S)$.
\end{definition}
Thus, a classifying map $p\colon E\rightarrow X$ for $S$ is a map in $S$ such that every other map in $S$ is a pullback of $p$ in an essentially unique way.

\begin{example}[{\cite[Prop.~6.1.6.3]{htt}}]
The class of all monomorphisms in an $\infty$-topos $\E$ has a classifying map.
\end{example}

\begin{proposition}[{\cite[Prop.~6.1.6.7]{htt}}]
\label{prop:locclassareclass}
Let $S$ be a local class of maps in an $\infty$-topos $\E$. Then, there are arbitrarily large regular cardinals $\kappa$ such that the class $S_\kappa$ of maps in $S$ that are relatively $\kappa$-compact (\cite[Def.~6.1.6.4]{htt}) is local and has a classifying map.
\end{proposition}

\begin{notation}
For every regular cardinal $\kappa$ as in \cref{prop:locclassareclass}, we denote by $\ti{\U_\kappa}\rightarrow\U_\kappa$ the classifying map for $\kappa$-compact maps in $\E$. Here, $\U$ stands for ``\emph{universe}'', a terminology borrowed from homotopy type theory (see \cite[\S~1.3]{hott}). If $S$ is a local class of maps, we denote by $\ti{\U_\kappa^S}\rightarrow\U_\kappa^S$ the classifying map for $S_\kappa$. Note that, by definition, there is a pullback square
\begin{equation}
\label{eq:classmappullbackofuniv}
\bfig
\node e(0,0)[\ti{\U_\kappa^S}]
\node m(400,0)[\ti{\U_\kappa}]
\node x(0,-250)[\U_\kappa^S]
\node y(400,-250)[\U_\kappa]

\arrow|a|[e`m;]
\arrow|l|[e`x;]
\arrow|r|[m`y;]
\arrow|b|[x`y;s_\kappa]

\place(70,-90)[\angle]
\efig
\end{equation}
where the map $s_\kappa$ is unique up to equivalence.
\end{notation}

\subsection{Objects of equivalences and univalence}
\label{sec:objofeqanduniv}
By \cite{univinloccarclosed}, for $X,Y\in\E$, there is an object $\Eq (X,Y)$ in $\E$ such that the global elements of $\Eq (X,Y)$ are the equivalences from $X$ to $Y$. Such objects of equivalences are used to define and characterize univalent maps in an $\infty$-topos.\par
Let $J(\E)$ be the \emph{core} of $\E$, that is, the (strict) pullback of $\infty$-categories
$$
\bfig
\node e(0,0)[J(\E)]
\node m(600,0)[\E]
\node x(0,-300)[J(\Ho(\E))]
\node y(600,-300)[\Ho(\E)]

\arrow|a|[e`m;]
\arrow|l|[e`x;]
\arrow|r|[m`y;]
\arrow|b|[x`y;]

\place(70,-90)[\angle]
\efig
$$
Since, for every $\infty$-category $\C$ and every $X,Y\in\C$, $\Ho(\C)(X,Y)\simeq \pi_0(\C(X,Y))$, it follows that there is a pullback square
\begin{equation}
\label{eq:J(E)(X,Y)aspllbck}
\bfig
\node e(0,0)[J(\E)(X,Y)]
\node m(800,0)[\E(X,Y)]
\node x(0,-300)[\pi_0(J(\E)(X,Y))]
\node y(800,-300)[\pi_0(\E(X,Y))]

\arrow|a|[e`m;]
\arrow|l|[e`x;]
\arrow|r|[m`y;]
\arrow|b|[x`y;]

\place(70,-90)[\angle]
\efig
\end{equation}
and that the map $J(\E)(X,Y)\rightarrow\E(X,Y)$ is a monomorphism.
\begin{proposition}[{\cite[Thm.~2.10]{univinloccarclosed}}]
For every $X,Y\in\E$, there is a subobject $\Eq_\E(X,Y)$ of $Y^X$ such that, for every $T\in\E$, there is an equivalence of $\infty$-groupoids
$$
\E(T,\Eq_\E(X,Y))\simeq J(\E_{/T})(X\times T,Y\times T),
$$
natural in $T\in\E$. Furthermore, this is also true ``locally'', that is, for every two objects $p,q$ in a slice category $\E_{/X}$.
\end{proposition} 
\begin{notation}
For $p\colon E\rightarrow X$ and $q\colon M\rightarrow X$, we write $\Eq_{/X}(E,M)$ for the domain of $\Eq_{\E_{/X}}(p,q)$. We will often just write $\Eq(p,q)$ for $\Eq_{\E_{/X}}(p,q)$. 
\end{notation}
We can give a more explicit description of $\Eq(X,Y)$ as follows. For $X,Y\in\E$, there is a map
$c_{X,Y}\colon X^Y\times Y^X\rightarrow X^X
$ 
obtained as the adjunct map to the composite
$$
\bfig
\morphism(0,0)<1100,0>[X^Y\times Y^X\times X`X^Y\times Y;X^Y\times \ev_{X,Y}]
\morphism(1100,0)<600,0>[X^Y\times Y`X;\ev_{Y,X}]
\efig
$$
\begin{lemma}\label{lm:charofeqxy}
For every $X,Y\in\E$, there is a pullback square
$$
\bfig
\node e(0,0)[\Eq(X,Y)]
\node m(800,0)[X^Y\times Y^X\times X^Y]
\node x(0,-300)[1]
\node y(800,-300)[X^X\times Y^Y]

\arrow|a|[e`m;]
\arrow|l|[e`x;]
\arrow|r|[m`y;(c_{X,Y}\circ \pr_{12},\ c_{Y,X}\circ\pr_{13})]
\arrow|b|[x`y;(\id_X,\id_Y)]

\place(70,-90)[\angle]
\efig
$$
where $\pr_{12}$ (resp. $\pr_{13}$) is the projection $Y^X\times X^Y\times X^Y\rightarrow Y^X\times X^Y$ onto the first two components (resp.~onto the first and last components). This is also true ``locally'' for every $p,q\in\E_{/X}$.
\end{lemma}
Note that, on global elements, the right vertical map sends $f\in\E(X,Y)$ and $g,h\in\E(Y,X)$ to $(gf,fh)\in \E(X,X)\times\E(Y,Y)$.
\begin{proof}
The statement for slice categories is proven exactly as the one for $\E$, so we just prove the latter. By \cite[Lemma 2.8]{univinloccarclosed}, there is an inversion map $i\colon \Eq(X,Y)\rightarrow \Eq(Y,X)$, such that, for every $T\in\E$, the following diagram commutes
$$
\bfig
\morphism(0,0)<1700,0>[J(\E_{/T})(X\times T,Y\times T)`1;]
\morphism(0,0)<0,-350>[J(\E_{/T})(X\times T,Y\times T)`J(\E_{/T})(X\times T,Y\times T)\times J(\E_{/T})(Y\times T,X\times T);(\id,i)]
\morphism(0,-350)<1700,0>[J(\E_{/T})(X\times T,Y\times T)\times J(\E_{/T})(Y\times T,X\times T)`J(\E_{/T})(X\times T,X\times T);]
\morphism(1700,0)|r|<0,-350>[1`J(\E_{/T})(X\times T,X\times T);\id_{X\times T}]
\morphism(0,-350)<0,-300>[J(\E_{/T})(X\times T,Y\times T)\times J(\E_{/T})(Y\times T,X\times T)`\E_{/T}(X\times T,Y\times T)\times \E_{/T}(Y\times T,X\times T);]
\morphism(0,-650)<1700,0>[\E_{/T}(X\times T,Y\times T)\times \E_{/T}(Y\times T,X\times T)`\E_{/T}(X\times T,X\times T);]
\morphism(1700,-350)<0,-300>[J(\E_{/T})(X\times T,X\times T)`\E_{/T}(X\times T,X\times T);]
\efig
$$
where the middle and bottom horizontal arrows are composition maps and the unlabelled vertical maps are monomorphisms. For each $T$, $i$ picks out an inverse for an equivalence $X\times T\rightarrow Y\times T$ over $T$. Now, let $P$ be the pullback in $\E$ of the cospan displayed in the statement of the lemma. If we still denote by $i$ the composite of $i$ with $\Eq(Y,X)\rightarrowtail X^Y$, there is a map
$
\Eq(X,Y)\rightarrow X^Y\times Y^X\times X^Y
$
given by $\Eq(X,Y)\rightarrowtail Y^X$ on the second component and by $i$ on the other two components. The above-mentioned result from \cite{univinloccarclosed} then implies that this map determines a morphism $\phi\colon\Eq(X,Y)\rightarrow P$.\par
To show that $\phi$ is an equivalence, we verify that $\E(T,\phi)$ is an equivalence of $\infty$-groupoids for every $T\in\E$. For ease of exposition, we show this for $T=1$ only; the same proof below goes through for any $T\in\E$ by working with $\E_{/T}(X\times T,Y\times T)$ rather than $\E(X,Y)$. There is a homotopy pullback in $\infty\Gpd$
$$
\bfig
\node e(0,0)[\E(1,P)]
\node m(1100,0)[\E(Y,X)\times\E(X,Y)\times\E(Y,X)]
\node x(0,-300)[1]
\node y(1100,-300)[\E(X,X)\times \E(Y,Y)]

\arrow|a|[e`m;]
\arrow|l|[e`x;]
\arrow|r|[m`y;]
\arrow|b|[x`y;]

\place(70,-100)[\angle]
\efig
$$ 
whereas $\E(1,\Eq(X,Y))\simeq J(\E)(X,Y)$ can be described as the (strict) pullback \eqref{eq:J(E)(X,Y)aspllbck}.\par 
Note that, by the description of equivalences in an $\infty$-category as those maps having a left and a right inverse, the composite map 
$$\E(1,P)\rightarrow\E(Y,X)\times\E(X,Y)\times\E(Y,X)\rightarrow \E(X,Y)$$
automatically lands in $J(\E)(X,Y)$, giving a map $\psi\colon\E(1,P)\rightarrow J(\E)(X,Y)$ which takes $(g,f,h)$ (for $f\colon X\rightarrow Y$, $gf\simeq \id$ and $fh\simeq \id$) to $f$. On the other hand, the map $\ti{\phi}\colon J(\E)(X,Y)\rightarrow\E(1,P)$ induced by $\phi$ sends an equivalence $f\colon X\rightarrow Y$ to $(i(f),f,i(f))$, for a chosen inverse $i(f)$ of $f$. It follows that $\psi\circ\ti{\phi}\simeq \id$. We conclude by observing that $\psi$ is an equivalence. For, if $f\in J(\E)(X,Y)$, by definition of $\E(1,P)$ we get that
\[
\hofib_f(\psi)\simeq \hofib_{\id_Y}((-)\circ f)\times \hofib_{\id_X}(f\circ (-))
\]
and both factors on the right are contractible, since $f$ is an equivalence.
\end{proof}

\begin{definition}[{\cite[$\S$3.1]{univinloccarclosed}}]
The \emph{object of equivalences} for $p\colon E\rightarrow X$ is the object of $\E_{/X\times X}$ given by 
$$
\Eq_{\E_{/X\times X}}(p\times \id_X,\id_X\times p)\colon\Eq_{/X\times X}(p\times \id_X,\id_X\times p)\rightarrow X\times X
$$
where $p\times \id_X\colon E\times X\rightarrow X\times X$ and similarly for $\mathrm{id}_{X}\times p$. We write the object of equivalences for $p$ as
$\Eq_{/X}(p)\colon\Eq_{/X}(E)\rightarrow X\times X$
\end{definition}
By the definition of $\Eq$, it follows that the identity map
$
\id_p\in J(\E_{/X})(p,p)
$
induces a map $\idtoequiv\colon X\rightarrow \Eq_{/X}(E)$ over $X\times X$.
\begin{definition}\cite[\S 3.2]{univinloccarclosed}
\label{def:univmap}
A \emph{univalent map} is a map $p\colon E\rightarrow X$ in $\E$ for which the associated map $\idtoequiv\colon X\rightarrow \Eq_{/X}(E)$ is an equivalence in $\E_{/X\times X}$.
\end{definition}
\begin{proposition}[{\cite[Prop.~3.8]{univinloccarclosed}}]\label{prop:localclassmapsareuniv}
Every classifying map $p$ is univalent.
\end{proposition}
We end this section with a result about truncated univalent map.
\begin{lemma}
\label{lm:univtuncmap}
Let $p\colon E\rightarrow X$ be a univalent and $n$-truncated map in an $\infty$-topos $\E$, for $n\geq (-1)$. Then both $E$ and $X$ are $(n+1)$-truncated.
\end{lemma}
\begin{proof}
It suffices to show that $X$ is $(n+1)$-truncated, that is, that $\Delta X$ is $n$-truncated. By univalence, $\Delta X\simeq\Eq_{/X}(p)$, and $\Eq_{/X}(p)$ is a subobject in $\E_{/X^2}$ of  $(\id_X\times p)^{(p\times \id_X)}$, so we can show that $(\id_X\times p)^{(p\times \id_X)}$ is $n$-truncated. Since $\id_X\times p$ is a pullback of $p$, it is an $n$-truncated object of $\E_{/X^2}$. But then $$(\id_X\times p)^{(p\times \id_X)}=\prod_{p\times\id_X} (p\times\id_X)^\ast{(\id_X\times p)}$$
is also $n$-truncated, because right adjoints preserve $n$-truncated objects (see \cite[Prop.~5.5.6.16]{htt}).
\end{proof}
\begin{corollary}
\label{cor:classmapofmono}
Let $p\colon E\rightarrow X$ be the classifying map of monomorphisms in an $\infty$-topos. Then $X$ is $0$-truncated and $E$ is contractible. In particular, $p$ is the subobject classifier of $\tau_{\leq 0}(\E)$, the ordinary $1$-topos of $0$-truncated objects of $\E$. 
\end{corollary}
\begin{proof}
We show that $E$ is contractible. By \cref{lm:univtuncmap} above, $X$ is $0$-truncated. Since $\id_1$ is a monomorphism, there is a pullback square
$$
\bfig
\node 1(0,0)[1]
\node e(400,0)[E]
\node 1b(0,-250)[1]
\node x(400,-250)[X]

\arrow|a|[1`e;\ast]
\arrow|b|[1b`x;\name{1}]
\arrow|l|[1`1b;\id_1]
\arrow|r|[e`x;p]

\place(70,-70)[\angle]

\efig
$$
for some map $\name{1}\colon 1\rightarrow X$, which is a monomorphism because $X$ is $0$-truncated. Since $E$ has a global element, the map $E\rightarrow 1$ is an effective epimorphism. The composite square
$$
\bfig
\node e1(-600,0)[E]
\node e2(-600,-300)[E]
\node 1(0,0)[1]
\node e(600,0)[E]
\node 1b(0,-300)[1]
\node x(600,-300)[X]

\arrow/->>/[e1`1;]
\arrow/->>/[e2`1b;]
\arrow|l|[e1`e2;\id_E]
\arrow|a|[1`e;\ast]
\arrow|b|[1b`x;\name{1}]
\arrow|l|[1`1b;\id_1]
\arrow|r|[e`x;p]

\place(70,-70)[\angle]
\place(-530,-70)[\angle]

\efig
$$
shows that $\id_E$ is the pullback of $p$ along $E\twoheadrightarrow 1\xrightarrow{\name{1}}X$. Because $p$ is a monomorphism, $\id_E$ is also the pullback of $p$ along itself. Since $p$ is a classifying map, we then get that the monomorphism $p$ has a (effective epi,mono)-factorization $E\twoheadrightarrow 1\rightarrowtail X$, so that $E\rightarrow 1$ has to be an equivalence, as needed.
\end{proof}
\begin{notation}
\label{notation:subobjclass}
Given \cref{cor:classmapofmono}, we follow the traditional convention in topos theory and denote by $t\colon 1\rightarrowtail \Omega$ the classifying map for monomorphisms in $\E$.
\end{notation}

\section{Reflective subfibrations and classifying maps}\label{sec:univfromreflsubf}
We introduce here the notion of \emph{reflective subfibrations} $L_\bullet$ on an $\infty$-topos $\E$ (\cite[\S A.2]{modinhott}). This is a collection of pullback-stable reflective subcategories $\D_{X}$ of $\E_{/X}$, with reflector $L_X$, as $X$ varies in $\E$. Objects of $\D_X$ are called $L$\emph{-local maps}.\par 
In \cref{sec:reflsubf}, we discuss some properties of reflective subfibrations. In \cref{sec:classofllocalmaps}, we show that the class of $L$-local maps form a local class of maps (\cref{prop:llocalmapsarelocal}), thus admitting a univalent classifying map (\cref{thm:llocalmapshaveaclassifyingmap}). We can use this observation to link reflective subfibrations on $\E$ to reflective subuniverses in homotopy type theory, as given in \cite{locinhott} and in \cite{modinhott}. 

\subsection{Reflective subfibrations}
\label{sec:reflsubf}
We give here the definition of reflective subfibrations on an $\infty$-topos $\E$ and derive some of their immediate properties.

\begin{definition}[{\cite[\S A.2]{modinhott}}]
\label{def:refsub}
Let $\E$ be an $\infty$-topos.
\begin{enumerate}
\item A \emph{system of reflective subcategories (srs)} $L_\bullet$ on $\E$ is the assignment, for each $X\in\E$, of an $\infty$-category $\D_X$ such that:
\begin{itemize}
\item Each $\D_X$ is a reflective $\infty$-subcategory of $\E_{/X}$, with associated localization functor $L_X=\colon \E_{/X}\rightarrow\E_{/X}$. This is the composite of the reflector of $\E_{/X}$ into $\D_X$ and the inclusion of $\D_X$ into $\E_{/X}$. When $X=1$, we write $\D$ for $\D_1$ and $L$ for $L_1$. 
\item For every map $f\colon X\rightarrow Y$ in $\E$, the pullback functor $f^\ast\colon\E_{/Y}\rightarrow\E_{/X}$ restricts to a functor $\D_Y\rightarrow\D_X$ which we still denote by $f^\ast$.
\end{itemize}
\item An srs $L_{\bullet}$ on $\E$ is a \emph{reflective subfibration} on $\E$, if, for any $f\colon X\rightarrow Y$ in $\E$ and any $p\in \E_{/Y}$, the induced map $L_X(f^\ast p)\rightarrow f^\ast(L_Y p)$ is an equivalence.
\item An srs $L_{\bullet}$ on $\E$ is \emph{composing} if, whenever $p\colon X\rightarrow Y$ is in $\D_Y$ and $q\colon Y\rightarrow Z$ is in $\D_Z$, the composite $qp$ is in $\D_Z$.
\item A \emph{modality} on $\E$ is a composing reflective subfibration $L_{\bullet}$ on $\E$.
\end{enumerate}  
\end{definition}
\begin{remark}
\label{rmk:localcharacterofreflsubf}
For every object $X\in\E$ and every map $f\colon Y\rightarrow X$, we have that $(\E_{/X})_{/f}\simeq \E_{/Y}$ (see \cite[Lemma~4.18]{joyalconjhott}). Therefore, for each $X\in\E$, a reflective subfibration $L_{\bullet}$ induces a reflective subfibration $L^{/X}_{\bullet}$ of $\E_{/X}$ by taking $\D^{/X}_{f}$ to be $\D_Y$. It follows that all the results we give below about reflective subfibrations on an $\infty$-topos also hold ``locally" in the $\infty$-topos $\E_{/X}$, for $X\in\E$.
\end{remark}
From now on, we fix a reflective subfibration $L_{\bullet}$ on our favorite $\infty$-topos $\E$. 

\begin{notation}\label{not:stuffaboutreflsub} We adopt the following notation for the rest of this work.

\begin{itemize}
\item A morphism $p\colon E\rightarrow X$ is called $L$-\emph{local} if, seen as an object of $\E_{/X}$, it is in $\D_X$. We call $E\in\E$ an $L$-\emph{local object} if $E\rightarrow 1$ is an $L$-local map. 
\item For $X\in\E$, $S_X$ denotes the class of all $L_X$-\emph{equivalences}, i.e., maps $\alpha$ in $\E_{/X}$ such that $L_X(\alpha)$ is an equivalence. Equivalently, $S_X=\prescript{\perp}{}{\D_X}$, where $\prescript{\perp}{}{\D_X}$ denotes the class of maps in $\E_{/X}$ which are left orthogonal to maps in $\D_X$. When it is clear that $\alpha$ is a map in $\E_{/X}$, we often drop the explicit reference to the object $X$, and just talk about $L$-\emph{equivalences}. 
\item Given $p\in\E_{/X}$, we write $\eta_X(p)\colon p\rightarrow L_X(p)$ for the reflection (or localization) map of $p$ into $\D_X$. Note that $\eta_X(p)\in S_X$. For $X\in\E$, we set $\eta(X):=\eta_1(X)$. 
\end{itemize}
\end{notation}
Given a map $f$ in $\E$, we denote by $\Sigma_f$ and by $\Pi_f$ the left and right adjoint to the pullback functor $f^\ast$, respectively. We will use the following remarks extensively.

\begin{lemma}
\label{lm:lequivandpullbdepsum}
Given $f\colon X\rightarrow Y$, we have:
\begin{itemize}
\item[(i)]$f^\ast(S_Y)\subseteq S_X$, that is, if $\alpha\colon p\rightarrow q$ is an $L_Y$-equivalence, then the induced map $f^\ast(p)\rightarrow f^\ast (q)$ on pullbacks is an $L_X$-equivalence;
\item[(ii)]$\Sigma_f(S_X)\subseteq S_Y$. 
\end{itemize}
\end{lemma}
\begin{proof}
Given $q\in\E_{/Y}$, $\eta_X(f^\ast q)=f^\ast (\eta_Y( q))$, so that, in particular, $f^\ast (\eta_Y (q))\in S_X$. Since, given a map $\alpha\colon q\rightarrow q'$ in $\E_{/Y}$, $L_Y(\alpha)$ is the unique map with $L_Y(\alpha)\circ \eta_Y (q)=\eta_Y (q')\circ\alpha$, the first claim follows immediately. The second claim follows by an adjunction argument: if $\alpha\in S_X$, given a map $\beta\colon r\rightarrow s$ in $\D_Y$, $\Sigma_f(\alpha)\perp_Y\beta\iff \alpha\perp_X f^\ast(\beta)$ and the latter holds since $f^\ast$ restricts to a functor $\D_Y\rightarrow\D_X$. 
\end{proof} 

\begin{proposition}\label{prop:propofreflsub}The following hold for any $X\in\E$.\begin{itemize}
\item[(i)]$p\in\E_{/X}$ is in $\D_X$ if and only if $\alpha\perp_X (p\rightarrow \id_X)$ for each $\alpha \in S_X$.
\item[(ii)] If $r\in\D_X$ and $f\colon X\rightarrow Y$ is any map in $\E$, then $\prod_f r$ is in $\D_Y$.
\item[(iii)]$\D_X$ is an exponential ideal in $\E_{/X}$, i.e., if $r\in\D_X$ and $p\in \E_{/X}$, then the internal hom $r^p$ is also in $\D_X$.
\item[(iv)]$L_X$ preserves products and $S_X$ is closed under products in $\E_{/X}$.
\item[(v)] A map $\alpha\colon p\rightarrow q$ is in $S_X$ if and only if, for each $r\in \D_X$, the map of internal homs $r^\alpha\colon r^q\rightarrow r^p$ is an equivalence.
\item[(vi)] $p\in\E_{/X}$ is in $\D_X$ if and only if $p^\alpha$ is an equivalence for each $\alpha\in S_X$. 
\end{itemize}
\end{proposition} 
\begin{proof}
For the first claim, since $S_X=\prescript{\perp}{}{\D_X}$, $\D_X\subseteq (S_X)^\perp$. On the other hand, if $p$ (that is, $p\rightarrow \id_X)$ is right orthogonal to $S_X$, then there is a map $\gamma\colon L_X(p)\rightarrow p$ with $\gamma\circ\eta_X(p)=\id_p$ and it is easy to see that $\eta_X(p)$ is then an equivalence.\par 
As for (ii), given $r\in\D_X$, $f\colon X\rightarrow Y$ and $\alpha\in S_Y$, adjointness gives that $\alpha\perp_Y\prod_f r\iff f^\ast (\alpha)\perp_X r$ and the latter orthogonality condition holds by \cref{lm:lequivandpullbdepsum} (i) and because $r$ is in $\D_X$. Since internal homs can be constructed via pullbacks and dependent products, it follows that $\D_X$ is an exponential ideal, establishing (iii). It is straightforward to check that this latter condition is equivalent to $L_X$ preserving and $S_X$ being closed under products in $\E_{/X}$, proving (iv).\par
As for (v), $\D_X$ being an exponential ideal implies that, for every $p\in\E_{/X}$ and $r\in\D_X$, $r^{\eta_X(p)}\colon r^{L_X(p)}\rightarrow r^{p}$ is an equivalence. Thus, if $\alpha\colon p\rightarrow q$ is in $S_X$, then $r^\alpha$ is equivalent in $\E^{\bullet\rightarrow\bullet}$ to the equivalence $r^{L_X(\alpha)}$. Conversely, if $r^\alpha$ is an equivalence for every $r\in\D_X$, then, given a map $\beta\colon r\rightarrow s$ in $\D_X$, consider the diagram
\begin{equation*}
\bfig
\node rq(-50,-100)[r^q]
\node sqxsprp(300,-300)[s^q\times_{s^p}r^p]
\node sq(300,-600)[s^q]
\node rp(800,-300)[r^p]
\node sp(800,-600)[s^p]

\arrow/{@{>}@/_15pt/}/[rq`sq;]
\arrow|a|/{@{>}@/^15pt/}/[rq`rp;r^\alpha]
\arrow[sqxsprp`sq;]
\arrow|a|[sqxsprp`rp;]
\arrow|b|[sq`sp;s^\alpha]
\arrow[rp`sp;]
\arrow||/-->/[rq`sqxsprp;]
\place(370,-390)[\angle]
\efig
\end{equation*}
Since $s^\alpha$ and $r^\alpha$ are equivalences, the dotted map is also such, which implies that $\alpha\perp_X\beta$. Finally, (vi) follows from (i), using closure under products of $S_X$.
\end{proof}
\begin{remark}
By \cref{prop:propofreflsub} (i) and (v), a map $\alpha$ in $\E_{/X}$ is such that, for every $r\in \D_X$, $\E_{/X}(\alpha,r)$ is an equivalence of $\infty$-groupoids if and only if $r^{\alpha}$ is an equivalence in $\E$. In this case, $\alpha\in S_X$. Similarly, $t\in \E_{/X}$ is in $\D_X$ if and only if, for every $\alpha\in S_X$, $\E_X(\alpha,t)$ is an equivalence, if and only if $t^\alpha$ is an equivalence. The external-hom description of $L$-equivalences is used in higher category theory, whereas the internal-hom description is the one available in homotopy type theory. (In fact, homotopy type theory can not even \emph{state} the external description, which provides some added value to the homotopy theoretic approach to localization.) Reflective subfibrations are defined so that these two perspectives coincide. 
 
\end{remark}

\begin{proposition}
\label{prop:weakcompproplocmap}
Suppose given composable maps $X\stackrel{f}{\rightarrow}Y\stackrel{g}{\rightarrow}Z$ in $\E$ such that $g,gf\in\D_Z$. Then $f\in\D_Y$. 
\end{proposition}
\begin{proof}
Let $\alpha\colon p\rightarrow q$ be a map in $\E_{/Y}$ which is an $L_Y$-equivalence. For $f$ to be in $\D_{Y}$, the induced map of $\infty$-groupoids
$
\E_{/Y}(\alpha,f)\colon\E_{/Y}(q,f)\rightarrow\E_{/Y}(p,f)
$
needs to be an equivalence. We will prove this by realizing such a map as the comparison map of fiber sequences in a pullback square. Consider the map $\Sigma_g (\alpha)\colon gp\rightarrow gq$ and let $\li{f}\colon gf\rightarrow g$ be the map induced by $f$. We have similar maps $\li{q}\colon gq\rightarrow g$ and $\li{p}\colon gp\rightarrow g$. By \cref{lm:lequivandpullbdepsum} (ii), $\Sigma_g(\alpha)$ is an $L_{Z}$-equivalence. Since both $g$ and $gf$ are in $\D_{Z}$ by hypothesis, the vertical maps in the commutative square
$$
\bfig
\node egqgf(0,0)[\E_{/Z}(gq,gf)]
\node egqg(1000,0)[\E_{/Z}(gq,g)]
\node egpgf(0,-300)[\E_{/Z}(gp,gf)]
\node egpg(1000,-300)[\E_{/Z}(gp,g)]

\arrow|a|[egqgf`egqg;\E_{/Z}(gq,\li{f})]
\arrow|l|[egqgf`egpgf;\E_{/Z}\left(\Sigma_g(\alpha),gf\right)]
\arrow|r|[egqg`egpg;\E_{/Z}\left(\Sigma_g(\alpha),g\right)]
\arrow|b|[egpgf`egpg;\E_{/Z}(gp,\li{f})]
\efig
$$
are equivalences. We can now take the induced map on fiber sequences. By the dual of \cite[Lemma 5.5.5.12]{htt}, we have:
$$
\hofib_{\li{q}}\left(\E_{/Z}(gq,\li{f})\right)=\left(\E_{/Z}\right)_{/g}(\li{q},\li{f})\simeq \E_{/Y}(q,f)
$$
As $\li{q}\left(\Sigma_g (\alpha)\right)=\li{p}$, we also get that 
$
\hofib_{\li{q}\left(\Sigma_g (\alpha)\right)}(\E_{/Z}(gp,\li{f}))\simeq\E_{/Y}(p,f)
$
and the induced map on fibers is $\E_{/Y}(\alpha, f)$.
\end{proof}
\begin{corollary}
If $X,Y$ are $L$-local objects, then any map $f\colon X\rightarrow Y$ is $L$-local. In particular, if $L$ is a modality and $X\in\D$, then $\D_{X}=\D_{/X}$.\qed
\end{corollary}

\begin{corollary}
\label{cor:localaresep}
If $g\colon A\rightarrow B$ is $L$-local, then so is $\Delta g\colon A\rightarrow A\times_B A$.
\end{corollary}
\begin{proof}
Since $g\in\D_B$, $g^\ast(g)\colon A\times_B A\rightarrow A$ is in $\D_A$. But $g^\ast(g)\circ\Delta g =\id_A$, so the claim follows by \cref{prop:weakcompproplocmap}, since $\id_A\in\D_A$. 
\end{proof}
\subsection{Classification of $L$-local maps}
\label{sec:classofllocalmaps}
We show here that, for a reflective subfibration $L_\bullet$ on an $\infty$-topos $\E$, the $L$-local maps form a local class of maps in $\E$ and, therefore, they admit a univalent classifying map.\par
Let $\M^L$ is the class of all $L$-local maps. $\M^L$ is stable under pullbacks, by \cref{def:refsub}.
\begin{lemma}
\label{lm:localmapsareclosedundercop}
$\M^L$ is closed under arbitrary small coproducts: if $I$ is a set and $f_j\in\D_{X_j}$ for $j\in I$, then $\coprod_j f_j$ is in $\D_{\left(\coprod_j X_j\right)}$.
\end{lemma}
\begin{proof}
For each $A\in\E$, $\id_A$ is $L$-local since it is the terminal object in $\E_{/A}$. In particular, $\id_0$ is an $L$-local map, where $0$ is the initial object of $\E$. This takes care of closure under empty coproducts. Since colimits in an $\infty$-topos are universal, there is an equivalence
$$
\E_{/\coprod_j X_j}\stackrel{\simeq}{\longrightarrow}\prod_j \E_{/X_j}
$$
given by taking pullbacks along the inclusions $\iota_j\colon X_j\rightarrow \coprod_j X_j$. It follows that, given a map $\alpha$ in $\E_{/\coprod_j X_j}$, $$\alpha\perp_{\coprod_j X_j} \coprod_j f_j\iff (\iota_k)^\ast(\alpha)\perp_{X_k} f_k\quad \text{for all\ } k\in I.$$
(Note that $(\iota_k)^\ast(\coprod_j f_j)=f_k$ because coproducts in $\E$ are disjoint.) The latter condition is true whenever $\alpha\in S_{\left(\coprod_j X_j\right)}$, thanks to \cref{lm:lequivandpullbdepsum} (i).
\end{proof}
\begin{lemma}
\label{lm:pllbckalongepi}
Given any pullback square in $\E$
$$
\bfig
\node e(0,0)[E]
\node m(400,0)[M]
\node x(0,-250)[X]
\node y(400,-250)[Y]

\arrow|a|[e`m;g]
\arrow|l|[e`x;p]
\arrow|r|[m`y;q]
\arrow|b|[x`y;f]

\place(70,-90)[\angle]
\efig
$$ 
where $f$ is an effective epimorphism, $p$ is in $\M^L$ if and only if $q$ is in $\M^L$.
\end{lemma}
\begin{proof}
By \cite[Lemma 6.2.3.16]{htt}, the statement is true if we replace ``being in $\M^L$" with ``being an equivalence". Suppose $p\in \D_X$ and consider $\eta_Y(q)\colon q\rightarrow L_Y(q)$. Since $f^\ast(\eta_Y(q))=\eta_X(f^\ast (q))=\eta_X(p)$ and $p\in\D_X$, $f^\ast(\eta_Y(q))$ must be an equivalence. By the opening observation, $\eta_Y(q)$ is also an equivalence, so that $q\in\D_Y$.  
\end{proof}

We have thus proved the following result.
\begin{proposition}
\label{prop:llocalmapsarelocal}
The class $\M^L$ of all $L$-local maps is a  local class of maps of $\E$.
\end{proposition}
We can use the above proposition to characterize reflective subfibrations on $\infty$-groupoids as fiberwise localizations. 
\begin{corollary}
\label{cor:reflsubfinspaces}
If $\E=\infty\Gpd$, a map $p\colon E\rightarrow X$ is $L$-local if and only if, for every $x\in X$, the homotopy fiber $\hofib_x(p)$ is an $L$-local $\infty$-groupoid.
\end{corollary}
\begin{proof}
If $\E=\infty\Gpd$, the canonical map $s\colon\coprod_{x\in X}1\longrightarrow X$ is an effective epimorphism since it induces a surjection on path components. Since colimits in an $\infty$-topos are universal, we have a pullback square
$$
\bfig
\node hf(-600,0)[]
\node h(0,0)[\coprod_{x\in X}\hofib_x(p)]
\node c(0,-300)[\coprod_{x\in X}1]
\node e(600,0)[E]
\node x(600,-300)[X]

\arrow[h`e;]
\arrow|b|[c`x;s]
\arrow[h`c;s']
\arrow|r|[e`x;p]
\place(70,-90)[\angle]
\efig
$$ 
where $s'$ is the coproduct of the maps $\hofib_x(p)\rightarrow 1$. Thus, $p$ is $L$-local if and only if $s'$ is $L$-local, by \cref{lm:pllbckalongepi}. Since, for $x_0\in X$, the pullback of $s$ along $x_0\colon 1\rightarrow \coprod_{x\in X}1$ is $\hofib_{x_0}(p)\rightarrow 1$, \cref{lm:localmapsareclosedundercop} and stability under pullbacks of $L$-local maps give us that $s'$ is $L$-local if and only if every $\hofib_{x_0}(p)$ is $L$-local.
\end{proof}

\begin{remark}
The above corollary can be generalized to any $\infty$-topos $\E$ upon replacing $\{1\}$ with a set $C$ of $\kappa$-compact generators of $\E$. In this case, by the same argument used in the proof of \cref{cor:reflsubfinspaces}, $p\colon E\rightarrow X$ is $L$-local if and only if, for every map $A \rightarrow X$ with $A\in C$, $A\times_X E\rightarrow A$ is $L$-local. If every object in $C$ is $L$-local and $L$ is a modality, this is the same as each object $A\times_X E$ being $L$-local.
\end{remark}

By \cref{prop:locclassareclass} and \cref{prop:localclassmapsareuniv}, we get the following result.

\begin{theorem}
\label{thm:llocalmapshaveaclassifyingmap}
Let $\kappa$ be a regular cardinal as in \cref{prop:locclassareclass} and let $\M^L_\kappa$ be the class of maps in $\E$ which are $L$-local and relatively $\kappa$-compact. Then $\M^L_\kappa$ has a classifying map
$u_\kappa^L\colon\ti{\U_\kappa^L}\rightarrow \U_\kappa^L$ which is univalent.
\end{theorem}
\begin{remark}
\label{rmk:reflsubanduniv}
If $u_\kappa\colon \ti{\U_\kappa}\rightarrow \U_\kappa$ classifies $\kappa$-compact maps in $\E$, then $u_\kappa^L$ is a pullback of $u_\kappa$ along a map $l_\kappa\colon\U_\kappa^L\rightarrow\U_\kappa$.
Since $u_\kappa$ and $u_\kappa^L$ are univalent, by \cite[Cor.~3.10]{univinloccarclosed} $l_\kappa$ is a monomorphism. Hence, $l_\kappa$ is the pullback of $t\colon 1\rightarrow\Omega$ (see \cref{notation:subobjclass}) along a unique map  $\IsLocal_\kappa\colon\U_\kappa\rightarrow~\Omega$, which then determines $u_\kappa^L$ (and all relatively $\kappa$-compact $L$-local maps). Maps $\U_\kappa\rightarrow\Omega$ are used to introduce $L$-local types in homotopy type theory, where $\U^L$ is called a \emph{subuniverse} of the \emph{universe} $\U$ (see \cite[Def.~2.1]{locinhott}). Note that, given a $\kappa$-compact object $X\in~\E$, the associated characteristic map $1\rightarrow \U_\kappa$ factors through $l_\kappa \colon \U_\kappa^L\rightarrowtail \U_\kappa$ (i.e., $X$ is $L$-local) if and only if the pullback of the composite 
$
1\longrightarrow\U_\kappa\xrightarrow{\IsLocal_\kappa} \Omega
$
along $1\rightarrow\Omega$ gives a $(-1)$-truncated object of $\E$ which is equivalent to $1$.
\end{remark}
\section{$L$-connected maps}
\label{sec:lconnmaps}
We study here properties of a class of maps associated with a reflective subfibration $L_\bullet$ on an $\infty$-topos $\E$, the $L$\emph{-connected maps}.\par 
\cref{sec:defandpropoflconnmaps} contains the main definitions, and a few technical properties. In \cref{section:stablefactaremod}, we prove that stable factorization systems on $\E$ correspond to \emph{modalities} $L_\bullet$ on $\E$ (\cref{thm:stablefactsystandmod}). In homotopy type theory, this correspondence is proven in \cite[\S 1]{modinhott}. Although some overlap between our proof and the one in \cite{modinhott} occurs, we did not follow the work there for our arguments. We conclude \cref{section:stablefactaremod} by discussing modalities on $\E$ associated to left exact reflective subcategories of $\E$.

\subsection{Definition and basic properties}
\label{sec:defandpropoflconnmaps}
Given a reflective subfibration $L_\bullet$ on $\E$, we define $L$-connected maps and prove some properties used in \cref{section:stablefactaremod}.
\begin{definition}
\label{def:lconnmaps}
$f\in\E_{/X}$ is said to be an $L$-\emph{connected map (in} $\E$\emph{)} if $L_X(f)\simeq \id_X$. Equivalently, $f$ is $L$-connected if 
$$(f\stackrel{\eta_X(f)}{\longrightarrow}L_X(f))\simeq (f\stackrel{f}{\rightarrow}\id_X)$$
in the arrow category of $\E_{/X}$, where the equivalence is given by $\id_f$ and $L_X(f)\rightarrow~\id_X$. We sometimes refer to this fact by saying that an $L$-connected map $f$ is \emph{its own reflection map}. 
\end{definition}

In particular, an $L$-connected map $f\colon E\rightarrow X$ is an $L_X$-equivalence when seen as a map $f\colon f\rightarrow \id_X$ in $\E_{/X}$.

\begin{remark}
\label{rmk:lconnmapsclosedunderpullbacks}
By taking the reflection of $f\in\E_{/X}$ into $\D_X$ and using stability under pullbacks of reflection maps (see \cref{def:refsub} (2)), it follows that $L$-connected maps are stable under pullbacks along arbitrary maps.
\end{remark}

\begin{lemma}
\label{lm:compoflconnislconn}
Suppose given composable maps $X\xrightarrow{f} Y\xrightarrow{g}Z $ and suppose that $f$ is $L$-connected. Then $g$ is $L$-connected if and only if $gf$ is $L$-connected.
\end{lemma}
\begin{proof}
Let $\eta_Z(g)\colon g\rightarrow L_Z(g)$ be the reflection map of $g\in\E_{/Z}$ into $\D_Z$. By hypothesis, $f\colon f\rightarrow \id_Y$ is the reflection map of $f$ into $\D_Y$. By \cref{lm:lequivandpullbdepsum} (ii), $\Sigma_g(f)\colon gf\rightarrow g$ is an $L_Z$-equivalence and so the composite map in $\E_{/Z}$ given by $\eta_Z(g)\circ \Sigma_g f\colon gf\rightarrow L_Z(g)$ is the reflection map of $gf$ into $\D_Z$. The claim follows. 
\end{proof}

\begin{lemma}
\label{lm:unitlconnforamodality}
If $L_\bullet$ is a modality on $\E$, then, for every map $f\colon E\rightarrow X$, the reflection map $\eta_X(f)\colon f\rightarrow L_X(f)$ is $L$-connected.
\end{lemma}
\begin{proof}
We prove this result for $X=1$, the general case having the same proof. Let $\eta(E)\colon E\rightarrow LE$ be the reflection map of $E$ and let $n\colon \eta(E)\rightarrow L_{LE}(\eta(E))$ be the reflection of $\eta(E)$ into $\D_{LE}$ (that is, $n=\eta_{LE}(\eta(E))$):
$$
\bfig
\node e(0,0)[E]
\node f(600,0)[LE]
\node lfe(900,300)[L_{LE}(E)]

\arrow|a|/{@{->}@/^20pt/}/[e`lfe;n]
\arrow|m|[e`f;\eta(E)]
\arrow|r|[lfe`f;L_{LE}(\eta(E))]
\efig
$$
By \cref{lm:lequivandpullbdepsum} and since $L$ is a modality, $n$ is an $L_1$-equivalence into the $L$-local object $L_{LE}(E)$, and it is therefore equivalent to $\eta(E)$ via the map $L_{LE}(\eta(E))$. Hence, $\eta (E)$ is $L$-connected.
\end{proof}

\begin{lemma}
\label{lm:orthpropoflconnandllocal}
Let $L_\bullet$ be a reflective subfibration on $\E$. Then the following hold.
\begin{itemize}
\item[(i)] Suppose $p\colon E\rightarrow X$ is a map in $\E$ with the property that $f\perp p$ for every $L$-connected map $f$. Then $p$ is an $L$-local map.
\item[(ii)] Suppose that $L_\bullet$ is a modality and let $f\colon A\rightarrow B$ be a map in $\E$ with the property that $f\perp p$ for every $L$-local map $p$. Then $f$ is an $L$-connected map.
\end{itemize}
\end{lemma}
\begin{proof}
We start by proving (i). Consider the reflection map of $p$ into $\D_X$ given by
$$
\bfig
\node e(0,0)[E]
\node lxe(500,0)[L_X (E)]
\node x(250,-250)[X]

\arrow|l|[e`x;p]
\arrow|a|[e`lxe;\eta_X (p)]
\arrow|r|[lxe`x;L_X(p)]
\efig
$$
Then, by \cref{lm:unitlconnforamodality}, $\eta_X(p)$ is $L$-connected. Therefore, by the hypothesis on $p$, there is a unique $n\colon L_X(E)\rightarrow E$ with $n\eta_X (p)=\id_E$ and $pn=L_X(p)$. In particular, $p$ is a retract of the $L$-local map $L_X(p)$ and it is therefore an $L$-local map itself.\\
As for (ii), consider the reflection map of $f$ into $\D_B$ given by
$$
\bfig
\node a(0,0)[A]
\node lba(500,0)[L_B (A)]
\node b(250,-250)[B]

\arrow|l|[a`b;f]
\arrow|a|[a`lba;\eta_B(f)]
\arrow|r|[lba`b;L_B(f)]
\efig
$$
The hypothesis on $f$ implies that there is a unique $s\colon B\rightarrow L_B(A)$ with $sf=\eta_B(f)$ and $L_B(f)s=\id_B$. In particular, $sL_B(f)$ can be seen as a map  $L_B(f)\rightarrow L_B(f)$ in $\E_{/B}$. Precomposing this map with $\eta_B(f)$, we deduce that $sL_B(f)=\id$. Hence, $s$ is an equivalence and $f$ is $L$-connected, by \cref{lm:unitlconnforamodality}.
\end{proof}

\subsection{Stable factorization systems are modalities}\label{section:stablefactaremod}
In \cite{agenBM} the term ``modality" is used as a synonym for a stable factorization system on an $\infty$-topos $\E$. We reconcile here that terminology with the definition of modality given in \cref{def:refsub}. Namely, we show that there is a bijective correspondence between stable factorization system on $\E$ and modalities on $\E$.\par

Recall that a factorization system $\ff=(\Ll,\Rr)$ on $\E$ is \emph{stable} if the left class $\Ll$ is stable under pullbacks. (The right class $\Rr$ is always stable under pullbacks.)  
\begin{example}
\label{ex:ntruncfactsyst}
For every $n\geq -2$, the $n$-truncated maps in an $\infty$-topos $\E$ form the right class of a stable factorization system, whose left class is given by the $n$-connected maps (see \cite[Prop.~3.3.6]{agenBM} and \cite[\S 6.5.1]{htt}).
\end{example}
Given a class $\M$ of maps in $\E$ and an object $X\in\E$, we let $\M_X$ be the class of maps in $\E_{/X}$ that are mapped into $\M$ by the forgetful functor $\E_{/X}\rightarrow \E$. 
\begin{lemma}[{\cite[Lemma 3.1.7]{agenBM}}]
\label{lemma:factsystonslices}
Let $\ff=(\Ll,\Rr)$ be a factorization system on $\E$. Then, for every $X\in\E$, $\ff_X:=(\Ll_X,\Rr_X)$ is a factorization system on $\E_{/X}$.
\end{lemma}
We are now ready to prove the following result.
\begin{theorem}
\label{thm:stablefactsystandmod}
Let $\E$ be an $\infty$-topos.
\begin{enumerate}
\item Let $\ff=(\Ll,\Rr)$ be a stable factorization system on $\E$. There exists a modality $L^\ff_\bullet$ on $\E$ whose local maps are exactly the maps in $\Rr$.
\item Let $L_\bullet$ be a modality on $\E$. Let $\Ll$ be the class of $L$-connected maps and $\Rr$ the class of $L$-local maps. Then $\ff_L=(\Ll,\Rr)$ is a stable factorization system on $\E$.
\end{enumerate}
\end{theorem}
\begin{proof}
We prove the two statements separately and we start from the first claim.\par 
Let $\ff=(\Ll,\Rr)$ be any factorization system on $\E$. Set $\D:=\Rr_{/1}$, the full subcategory of $\E$ spanned by those $X\in\E$ such that the map $X\rightarrow 1$ is in $\Rr$. It follows from uniqueness and functoriality of the $(\Ll,\Rr)$-factorizations (see \cite[\S~3.1]{agenBM} and \cite[Prop.~5.2.8.17]{htt}) that $\D$ is a reflective subcategory of $\E$. For $X\in\E$, the value $L(X)\in\D$ of the reflector and the unit map $\eta(X)\colon X\rightarrow L(X)$ are determined by the fact that $(\eta(X),\ L(X)\rightarrow 1)$ is the $(\Ll,\Rr)$-factorization of the map $X\rightarrow 1$. In particular, $\eta (X)$ is a map in $\Ll$, which gives the universal property for the unit map. We can apply the same considerations to the factorization system $(\Ll_X,\Rr_X)$ on $\E_{/X}$, and obtain a reflective subcategory $\D_{X}:=(\Rr_X)_{/\id_{X}}$ of $\E_{/X}$, for every $X\in\E$. Note that $p\in\E_{/X}$ is in $\D_X$ if and only if it is in $\Rr$ when considered as a map in $\E$. Since the class $\Rr$ is closed under pullbacks along any map and under compositions with maps in $\Rr$ (see \cite[Lemma 3.1.6 (3)]{agenBM}), it follows that the assignment $X\mapsto \D_{X}$ so defined gives rise to a composing srs $L_\bullet^{\ff}$ on $\E$ (see \cref{def:refsub}). It is straightforward to see that $L_\bullet^{\ff}$ is a reflective subfibration when $\ff=(\Ll,\Rr)$ is a stable factorization system.\par
For the second claim, let $L_\bullet$ be a modality on $\E$ and let $\ff_L=(\Ll,\Rr)$ be as in the statement of the theorem. For any $f\colon E\rightarrow X$ in $\E$, the reflection of $f$ into $\D_X$ is an $\ff_L$-factorization of $f$, by \cref{lm:unitlconnforamodality}. Both $\Ll$ and $\Rr$ contain all equivalences and are closed under composition, by \cref{lm:compoflconnislconn} and because $L_\bullet$ is a modality. Furthermore, by \cref{rmk:lconnmapsclosedunderpullbacks}, the left class is closed under pullbacks, while \cref{lm:orthpropoflconnandllocal} says that $\prescript{}{}{\Ll}^{\perp}\subseteq \Rr$ and $\prescript{\perp}{}{\Rr}\subseteq \Ll$.\par 
Thus, to conclude that $\ff_L$ is a factorization system, we just need to show that, for every $L$-connected map $f\colon X\rightarrow ~Y$ and for every $L$-local map $p\colon E\rightarrow Z$, we have that $f\perp p$, that is the following commutative diagram in $\infty\Gpd$
$$
\bfig
\node eye(0,0)[\E(Y,E)]
\node eyz(700,0)[\E(Y,Z)]
\node exe(0,-300)[\E(X,E)]
\node exz(700,-300)[\E(X,Z)]

\arrow|a|[eye`eyz;\E(Y,p)]
\arrow|l|[eye`exe;\E(f,E)]
\arrow|r|[eyz`exz;\E(f,Z)]
\arrow|b|[exe`exz;\E(X,p)]
\efig
$$ 
is a pullback square. Equivalently, we can check that the induced map on fibers is an equivalence. By looking at the fiber over $k\in\E(Y,Z)$, such an induced map is given by 
$
\E_{/Z}(\li{f},p)\colon\E_{/Z}(k,p)\rightarrow\E_{/Z}(kf,p),
$
where $\li{f}$ is given by considering $f$ as a map $kf\rightarrow k$ in $\E_{/Z}$. This map fits into the following commutative square in $\infty\Gpd$
$$
\bfig
\node elzkp(0,0)[\E_{/Z}(L_Z(k),p)]
\node ekp(1100,0)[\E_{/Z}(k,p)]
\node elzkfp(0,-300)[\E_{/Z}(L_Z(kf),p)]
\node ekfp(1100,-300)[\E_{/Z}(kf,p)]

\arrow|a|[elzkp`ekp;\E_{/Z}(\eta_Z(k),p)]
\arrow|l|[elzkp`elzkfp;\E_{/Z}(L_Z(\li{f}),p)]
\arrow|r|[ekp`ekfp;\E_{/Z}(\li{f},p)]
\arrow|b|[elzkfp`ekfp;\E_{/Z}(\eta_Z(kf),p)]
\efig
$$
(Here, $\eta_Z(k)\colon k\rightarrow L_Z(k)$ is the reflection of $k$ into $\D_Z$ and similarly for $\eta_Z(kf)$.) Note that the horizontal maps are equivalences as $p$ is $L$-local. Since $f$ is $L$-connected, $f\colon f\rightarrow \id_Y$ is an $L_Y$-equivalence and so $\li{f}=\Sigma_k (f)$ is an $L_Z$-equivalence, by \cref{lm:lequivandpullbdepsum} (ii). Therefore, $L_Z(\li{f})$ is an equivalence and then $\E_{/Z}(\li{f},p)$ is also an equivalence. Hence, $f\perp g$ and $\ff_L$ is a stable factorization system. 
\end{proof}

\begin{corollary}
\label{cor:stabtomodandmodtostabbij}
The assignments $\ff\mapsto L^\ff_\bullet$ and $L_\bullet\mapsto \ff_L$ determine a bijective correspondence between the class of stable factorization systems on an $\infty$-topos $\E$ and the class of collections $\{\D_X\}_{X\in\E}$ of reflective subcategories $\D_X\subseteq\E_{/X}$ which form the $L$-local maps for a modality $L_\bullet$ on $\E$.
\end{corollary}
\begin{proof}
If $\ff=(\Ll,\Rr)$ is a stable factorization system, the right class of $\ff_{L^\ff_\bullet}$ is again $\Rr$ and then the left class has to be $\Ll$ since $\Ll=\prescript{\perp}{}{\Rr}$. Thus, $\ff=\ff_{L^\ff_\bullet}$. If $L_\bullet$ is the modality associated to $\{\D_X\}_{X\in \E}$, then, by definition, the reflective subcategories $\ti{\D}_X$ associated to the modality $L^{\ff_L}_\bullet$ are given by the $L$-local maps with codomain $X\in \E$. Therefore, $\ti{\D}_X=\D_X$, for every $X\in\E$.
\end{proof}
\begin{example}
\label{ex:ntruncmod}
By \cref{ex:ntruncfactsyst}, for every $n\geq -2$, there is a modality $L_\bullet^n$ on $\E$, for which the $L$-local maps are the $n$-truncated maps. We call this modality the $n$-\emph{truncated} modality on $\E$.
\end{example}

We conclude this section by applying \cref{thm:stablefactsystandmod} to construct reflective subfibrations out of a left exact localization of an $\infty$-topos $\E$. Suppose $\D\subseteq\E$ is a reflective subcategory, with reflector $a\colon\E\rightarrow\D$. Recall that $\D$ is called \emph{left exact} if $a$ is left exact, that is, if it preserves finite limits.

\begin{proposition}
\label{prop:leftexactreflsubmod}
Let $i\colon\D\hookrightarrow\E$ be a left exact reflective subcategory of $\E$ with reflector $a\colon \E\rightarrow\D$. Set $L:=ia\colon \E\rightarrow\E$. Then there is a modality $L_\bullet$ on $\E$ for which $L_1=L$, and $f\colon X\rightarrow Y$ is $L$-local if and only if there is a pullback square
\begin{equation}
\label{eq:conditionforleftexactloc}
\bfig
\node x(0,0)[X]
\node lx(400,0)[LX]
\node y(0,-250)[Y]
\node ly(400,-250)[LY]

\arrow|a|[x`lx;\eta(X)]
\arrow|l|[x`y;f]
\arrow|r|[lx`ly;Lf]
\arrow|b|[y`ly;\eta(Y)]
\efig
\end{equation} 
\end{proposition} 
\begin{proof}
Since $L=ia$ is left exact, $L$ gives rise to a stable factorization system on $\E$, by \cite[Lemma 2.6.4]{goodwillie}. The left class $\Ll$ of this factorization system consists of the $L$-equivalences and $\Rr=\Ll^\perp$ is exactly the class of all maps $f\colon X\rightarrow Y$ in $\E$ satisfying the stated pullback condition. We conclude by \cref{thm:stablefactsystandmod} (1). 
\end{proof}
In \cite[Prop.~3.6]{l'loc}, we show that a pullback-like characterization of $L$-local maps similar to the above one can be given for any reflective subfibration $L_\bullet$ on $\E$, upon suitably replacing the reflection map $\eta (Y)$.
\begin{remark}
In the context of \cref{prop:leftexactreflsubmod}, \cref{cor:stabtomodandmodtostabbij} implies that the $L$-connected maps are exactly the $L_1$-equivalences. Therefore, \cref{prop:leftexactreflsubmod} is really just a special case of the constructions given in \cref{thm:stablefactsystandmod} with a different description of the class of $L$-connected and $L$-local maps. In fact, we can note the following. Recall from \cref{def:lconnmaps} that every $L$-connected map $f\colon Y\rightarrow X$ is an $L_X$-equivalence when seen as a map $f\colon f\rightarrow\id_X$. Using \cref{lm:lequivandpullbdepsum}~(ii), it follows that, for the modality $L_\bullet$ of \cref{prop:leftexactreflsubmod}, the following hold. 
\begin{itemize}
\item[(a)] The class of $L_1$-equivalences and the class of $L$-connected maps coincide. 
\item[(b)] For every $X\in\E$, a map $\alpha\colon p\rightarrow q$ is an $L_X$-equivalence if and only if $\Sigma_X(\alpha)$ is an $L_1$-equivalence.
\end{itemize} 
The modalities on $\E$ with these properties correspond to the so-called \emph{lex modalities} in homotopy type theory --- see \cite[Thm.~3.1]{modinhott}.
\end{remark}

\section{Existence of $f$-localization}\label{sec:localizwrtmap} Fix a map $f\colon A\rightarrow B$ in an $\infty$-topos $\E$. We will construct here a reflective subfibration on $\E$ out of this datum. When $\E=\infty\Gpd$, this recovers the classical localization of spaces at a map $f$.

\begin{definition}
A map $f$ is \emph{internally orthogonal} to a map $g$ in $\E$ if $(f\times Z)\perp g$ for every object $Z\in\E$. If this holds, we write $f\iperp g$.
\end{definition}
\begin{definition}
\label{def:cartfactsyst}
A factorization system $(\Ll,\Rr)$ in an $\infty$-topos $\E$ is called \emph{cartesian} if, for every $l\in\Ll$ and every $r\in\Rr$, $l\iperp r$ (rather than simply having $l\perp r$).
\end{definition}
Let $\Rr:=\{f\}^{\intperp}$ be the class of all maps $g$ such that $f\iperp g$. Then, if we let  $\Ll:=\prescript{\perp}{}{\Rr}=\prescript{\intperp}{}{\Rr}$, $(\Ll,\Rr)$ is a cartesian factorization system (see \cite[Prop.~3.2.9]{agenBM}). Thus, as in \cref{thm:stablefactsystandmod}, we get that $\D:=\Rr_{/1}$ is a reflective subcategory of $\E$. Since, if $l,l'\in\Ll$, then $l\times l'\in \Ll$, it follows that $\D$ consists of all those $X\in\E$ such that $X^f\colon X^B\rightarrow X^A$ is an equivalence in $\E$, and that $\D$ is an exponential ideal (see \cref{prop:propofreflsub} (iii) for terminology).
\begin{definition}
An object $X\in\E$ such that $X^f\colon X^B\rightarrow X^A$ is an equivalence is called an $f$-\emph{local} object. 
\end{definition}
\begin{remark}
\label{rmk:internallyflociffexternallyflocal}
If we let $L^f\colon \E\rightarrow \E$ be the localization functor associated to $\D$, the $L^f$-equivalences form the class $\prescript{\perp}{}{\D}=\prescript{\iperp}{}{\D}$. Since $X\in\E$ is $f$-local if and only if $X^s$ is an equivalence for every $L^f$-equivalence $s$, one can see that $X\in\E$ is $f$-local if and only if the map of $\infty$-groupoids $\E(f,X)\colon\E(B,X)\rightarrow\E(A,X)$ is an equivalence.
\end{remark}
For any fixed $X\in\E$, we can consider the map $f\times X\colon \pr_A\rightarrow \pr_B$ in $\E_{/X}$, where $\pr_A\colon A\times X\rightarrow X$ is the projection map, and similarly for $\pr_B$. Set $\Rr^X:=\{f\times X\}^{\intperp}$. As above, we get a cartesian factorization system $(\Ll^X,\Rr^X)$ in $\E_{/X}$ with $\Ll^X :=\prescript{\perp}{}{\left(\Rr^X\right)}=\prescript{\intperp}{}{\left(\Rr^X\right)}$. Then $\D_{X}:=\Rr^X_{/\id_{X}}$ is a reflective subcategory of $\E_{/X}$.

\begin{definition}
\label{def:flocalmap}
We call a map $p\colon E\rightarrow X$ in $\E$ an $f$-\emph{local map} if the map $p^{(f\times X)}\colon p^{\pr_B}\rightarrow p^{\pr_A}$ in $\E_{/X}$ is an equivalence.
\end{definition}

The reflective subcategory $\D_{X}$ consists precisely of the $f$-local maps.
\begin{remark}
\label{rmk:r^xcontainedinr_x}
Using the fact that, if $s\in\E_{/X}$ and $S=\dom (s)$, the product map $(f\times X)\times^X s$ in $\E_{/X}$ is the map 
$$
\bfig
\Vtriangle|alr|<300,250>[A\times S`B\times S`X;f\times S`s\pr_A`s\pr_B]
\efig
$$
in $\E_{/X}$, it is easy to see that $\Rr_X\subseteq \Rr^X$, where --- as in \cref{lemma:factsystonslices} --- $\Rr_X$ is the class of maps in $\E_{/X}$ that are in $\Rr$ when seen as maps in $\E$.
\end{remark}

We claim that this assignment $X\mapsto\D_X$ defines a reflective subfibration. 
\begin{proposition}
\label{prop:flocreflsubf}
For every map $f\colon A\rightarrow B$ in $\E$, there is a reflective subfibration $L_\bullet^f$ on $\E$ for which the local maps are exactly the $f$-local maps.
\end{proposition}
\begin{proof}
Fix $f\colon A\rightarrow B$ in $\E$. We show that the reflective subcategories $\D_X$ of $f$-local maps constructed above give rise to a reflective subfibration. Consider a map $g\colon Y\rightarrow X$ in $\E$. We show that $g^{\ast}(p)$ is in $\D_Y$ for $p\in\D_X$. By hypothesis, $p^{(f\times X)}$ is an equivalence in $\E_{/X}$, so that $g^{\ast} (p^{f\times X})$ is an equivalence in $\E_{/Y}$. But $g^{\ast} \left(p^{(f\times X)}\right)$ is equivalent to $\left( g^{\ast} (p)\right)^{g^{\ast}(f\times X)}$ (as maps in $\E_{/Y}$). Since $g^{\ast}(f\times X)=f\times Y$ as maps in $\E_{/Y}$, we can conclude that $g^{\ast} (p)$ is in $\D_Y$.\par
We now verify the condition of \cref{def:refsub} (2). Let $p\in\E_{/X}$ with $E:=\dom(p)$. The reflection of $p$ into $\D_X$ is given by the $(\Ll^X,\Rr^X)$-factorization of $p\rightarrow\id_X$, which we depict as the diagram
$$
\bfig
\node e(0,0)[E]
\node lxe(500,0)[L_X(E)]
\node x(1000,0)[X]

\arrow|b|[e`lxe;l_p]
\arrow|b|[lxe`x;r_p]
\arrow|a|/{@{>}@/^1em/}/[e`x;p]
\efig
$$
in $\E$. We need to show that, for $g\colon Y\rightarrow X$ in $\E$, $(g^\ast(l_p),g^\ast (r_p))$ is the $(\Ll^Y,\Rr^Y)$-factorization of $g^\ast (p)\rightarrow\id_Y$ in $\E_{/Y}$, where $g^\ast(l_p)\colon g^\ast(p)\rightarrow g^\ast (r_p)$. Note that, by the first part above, we have $g^\ast (r_p)\in\D_Y$ (that is, $g^\ast (r_p)\rightarrow \id_Y$ is in $\Rr^Y$). Thus, we need to show that $g^\ast(l_p)$ is in $\Ll^Y$. This means showing that, for every $m\in\Rr^Y$, $g^\ast (l_p)\perp_Y m$. But, by adjointness, this orthogonality condition in $\E_{/Y}$ is equivalent to the orthogonality condition $l_p\perp_X \prod_g m$ in $\E_{/X}$. Since $l_p\in\Ll^X$ by hypothesis, it suffices to show that, for every $g\colon Y\rightarrow X$ and every $m\in\Rr^Y$, $\prod_g m$ is in $\Rr^X$.\par 
By \cref{rmk:r^xcontainedinr_x}, if $s\in\E_{/X}$ and $S:=\dom (s)$, the product map $(f\times X)\times^X s$ in $\E_{/X}$ is the map $f\times S\colon s\pr_A\rightarrow s\pr_B$ in $\E_{/X}$. By definition, $\prod_g m$ is in $\Rr^X$ precisely if $(f\times S)\perp_X \prod_g m$ in $\E_{/X}$ for every $s\in\E_{/X}$ as above. By adjointness, this happens if and only if $g^\ast(f\times S)\perp_Y m$ in $\E_{/Y}$. An easy application of the pasting lemma for pullbacks shows that, if we denote the domain of $g^\ast(s)$ by $g^\ast (S)$, $g^\ast(f\times S)$ is the map
$$
\bfig
\Vtriangle|alr|<400,250>[A\times g^\ast (S)`B\times g^\ast(S)`Y;f\times g^\ast(S)`\pr_A`\pr_B]
\efig
$$
in $\E_{/Y}$. This map is the product map of the object $g^\ast (s)\in\E_{/Y}$ with the map $f\times Y\colon\pr_A\rightarrow\pr_B$ in $\E_{/Y}$. Since $m\in\Rr^Y=\{f\times Y\}^{\intperp}$, we can conclude that $g^\ast(f\times S)\perp_Y m$, as required. 
\end{proof}
\begin{definition}
Given a map $f\colon A\rightarrow B$ in $\E$, we call the reflective subfibration $L^f$ of \cref{prop:flocreflsubf}, the \emph{f-local reflective subfibration} on $\E$. When $f$ is the unique map $A\rightarrow 1$ for an object $A\in\E$, we call the $f$-local reflective subfibration $A$-\emph{nullification}.
\end{definition}

\begin{proposition}[cf.~{\cite[Ex.~3.4.3]{agenBM}}]
\label{prop:nullisamod}
$A$-nullification is a modality for every $A\in\E$.
\end{proposition}
\begin{proof}
Using the same notation as in \cref{lemma:factsystonslices}, we show that, for every $X\in\E$, $\Rr^X=\Rr_X$ so that the claim follows from \cref{thm:stablefactsystandmod}. By \cref{rmk:r^xcontainedinr_x}, $\Rr_X\subseteq\Rr^X$ holds for every $f$-local reflective subfibration so we just need to prove the reverse inclusion. Let then $\alpha\colon p\rightarrow q$ be a map in $\Rr^X$ and set $m:=\Sigma_X (\alpha)$. Showing that $\alpha$ is in $\Rr_X$ means proving that $m\in\Rr$. That is, we have to show that, for every $C\in\E$, every commutative square in $\E$
\begin{equation}
\tag{$\ast$}
\bfig
\square<400,250>[A\times C`E`C`M;h`\pr_A`m`k]
\efig
\end{equation}
has a unique diagonal filler. If we let $\li{\pr}_A\colon k\pr_A\rightarrow k$ be the map in $\E_{/M}$ induced by $\pr_A$, this means showing that, for every $k\in\E(C,M)$, in the following comparison diagram of fiber sequences
$$
\bfig
\node emkm(0,0)[\E_{/M}(k,m)]
\node ece(1100,0)[\E(C,E)]
\node ecm(2200,0)[\E(C,M)]
\node emkprm(0,-400)[\E_{/M}(k\pr_A,m)]
\node eace(1100,-400)[\E(A\times C,E)]
\node eacm(2200,-400)[\E(A\times C,M)]

\arrow[emkm`ece;]
\arrow[ece`ecm;\E(C,m)]
\arrow|b|[emkprm`eace;]
\arrow|b|[eace`eacm;\E(A\times C,m)]
\arrow|m|[emkm`emkprm;\E_{/M}(\li{\pr}_A,m)]
\arrow|m|[ece`eace;\E(\pr_A,E)]
\arrow|m|[ecm`eacm;\E(\pr_A,M)]
\efig
$$
the rightmost square is a pullback, that is, the leftmost map is an equivalence. Now, $k$ gives rise to a map $\li{k}\colon qk\rightarrow q$ in $\E_{/X}$ and
$$
\hofib_{\li{k}}\left(\E_{/X}(qk,\alpha)\right)=(\E_{/X})_{/q}(\li{k},\alpha)\simeq \E_{/M}(k,m)
$$
Similarly, $k\pr_A$ gives rise to a map $\li{k\pr}_A\colon qk\pr_A\rightarrow q$ in $\E_{/X}$ and $$\hofib_{\li{k\pr}_{A}}\left(\E_{/X}(qk\pr_A,\alpha)\right)\simeq\E_{/M}(k\pr_A,m).$$
It follows that there is a comparison diagram of fiber sequences
$$
\bfig
\node emkm(0,0)[\E_{/M}(k,m)]
\node eqkp(900,0)[\E_{/X}(qk,p)]
\node eqkq(2000,0)[\E_{/X}(qk,q)]
\node emkprm(0,-400)[\E_{/M}(k\pr_A,m)]
\node eqkprp(900,-400)[\E_{/X}(qk\pr_A,p)]
\node eqkprq(2000,-400)[\E_{/X}(qk\pr_A,q)]

\arrow[emkm`eqkp;]
\arrow[eqkp`eqkq;\E_{/X}(qk,\alpha)]
\arrow[emkprm`eqkprp;]
\arrow[eqkprp`eqkprq;\E_{/X}(qk\pr_A,\alpha)]
\arrow|m|[emkm`emkprm;\E_{/M}(\li{\pr}_A,m)]
\arrow|m|[eqkp`eqkprp;\E_{/X}(\li{\pr}_A,p)]
\arrow|m|[eqkq`eqkprq;\E_{/X}(\li{\pr}_A,q)]
\efig
$$
We claim that the right square is a pullback so that the induced map on fibers is an equivalence. Indeed, $\li{\pr}_A$ is the product map in $\E_{/X}$ of $\id_{qk}$ with $\pr_2\colon \pr_2\rightarrow \id_X$, where $\pr_2\colon A\times X\rightarrow X$. Since $\Ll^X$ is closed under products and $\pr_2,\id_{qk}\in\Ll^X$, $\li{\pr}_A\perp_X \alpha$, which means that the right square above is a pullback.
\end{proof}
\begin{definition}
Let $S=\{f_i\colon A_i\rightarrow B_i\}_{i\in I}$ be a set of maps in $\E$. A map $p$ is \emph{S-local} if it is $f_i$-local for every $i\in I$.
\end{definition}

\begin{proposition}
\label{prop:existofsloc}
Let $S=\{f_i\colon A_i\rightarrow B_i\}_{i\in I}$ be a set of maps in $\E$. There is a reflective subfibration $L^S_\bullet$ on $\E$ whose local maps are the $S$-local map. In particular, a map is an $L^S$-equivalence precisely if it is an $f_i$-equivalence for every $i\in I$.
\end{proposition}
\begin{proof}
\cite[Prop.~3.2.9]{agenBM} gives a cartesian factorization system $(\Ll,\Rr)$ on $\E$ in which $\Rr=S^\intperp$. Furthermore, $X\in\E$ belongs to the associated reflective subcategory $\D=\Rr_{/1}$ if and only if $X^{f_i}$ is an equivalence for every $i\in I$. For every $X\in\E$, we set $S\times X:=\{f_i\times X\colon \pr_{A_i}\rightarrow \pr_{B_i}\}_{i\in I},$ consider the associated cartesian factorization system on $\E_{/X}$, and obtain a reflective subcategory $\D_X$ of $\E_{/X}$. An argument essentially the same as the proof of \cref{prop:flocreflsubf} shows that in this way we get a reflective subfibration $L^S_\bullet$ on $\E$ with the required properties.
\end{proof}

\section{$L$-separated maps}
\label{sec:lsepmaps}
Given a reflective subfibration $L_\bullet$ on $\E$, we introduce $L$\emph{-separated maps}, that is, those maps whose diagonal is an $L$-local map. We prove some closure properties of $L$-separated maps. We leave to the companion paper \cite{l'loc} the delicate task of showing that there is a reflective subfibration $L'_\bullet$ on $\E$ such that the $L'$-local maps are exactly the $L$-separated maps. We will, however, gather here some results from \cite{l'loc} that we need. 

\begin{definition}
\label{def:lseparatedmaps}
A map $p\colon E\rightarrow X$ in $\E$ is called $L$-\emph{separated} or $L'$-\emph{local} if the object $\Delta p\in\E_{/E\times_{X}E}$ is in $\D_{E\times_{X}E}$, i.e., if $\Delta p$ is an $L$-local map.
\end{definition}

\begin{remark}
Given a Kan complex $X$, $\Delta X$ is, up to equivalence, the path-fibration $X^{\Delta [1]}\twoheadrightarrow X\times X$. Hence, \cref{def:lseparatedmaps} describes all those spaces $X$ for which the fibers of the path-fibration map (i.e., the spaces $\Path(x,y)$ of paths in $X$ between any two points $x,y\in X$) are $L$-local.
\end{remark}

\begin{remark}
\label{rmk:llocarandmonoarelsep}
We can make the following elementary observations.
\begin{itemize}
\item[(i)] \cref{cor:localaresep} is exactly the statement that every $L$-local map is $L$-separated.
\item[(ii)] The diagonal of every monomorphism is an equivalence, so every monomorphism is $L$-separated. In particular, $(-1)$-truncated objects are $L$-separated.
\end{itemize}
\end{remark}
\begin{example}
\label{ex:n+1truncnsep}
For every $n\geq -2$, consider the $n$-truncated modality $L=L^n_\bullet$ on $\E$ of \cref{ex:ntruncmod}. By \cite[Lemma 5.5.6.15]{htt}, the $L^n$-separated maps are the $(n+1)$-truncated maps. 
\end{example}
$L$-separated maps share the same closure properties as $L$-local maps. 
\begin{proposition}
\label{prop:closureoflsepmapsunderpllbckdepprod}
Let $f\colon Y\rightarrow X$ be a map in $\E$, and let $p\colon E\rightarrow X$ and $q\colon M\rightarrow~Y$ be $L$-separated maps. Then $f^{\ast}(p)\in\E_{/Y}$ and $\prod_f q\in\E_{/X}$ are $L$-separated. Furthermore, the internal hom $p^f$ is $L$-separated.
\end{proposition}
\begin{proof}
We begin by showing that $f^\ast (p)$ is $L$-separated. Write $f^\ast (E)$ for $Y\times_X E$, so that $f^\ast (p)\colon f^\ast (E)\rightarrow Y$. The composite pullback square in $\E$
$$
\bfig
\hSquares(0,0)|aallrbb|%
<300>[f^\ast (E)\times_Y f^\ast (E)`f^\ast (E)`E`f^\ast (E)`Y`X;```f^{\ast}(p)`p`f^{\ast}(p)`f]
\place(70,180)[\angle]
\place(900,180)[\angle]
\efig
$$
is the same as the composite square
$$
\bfig
\hSquares(0,0)|aallrbb|%
<300>[f^\ast (E)\times_Y f^\ast (E)`E\times_X E`E`f^\ast (E)`E`X;```p^{\ast}(p)`p``p]
\place(950,200)[\angle]
\efig
$$ 
Therefore, in this last diagram, the left square is a pullback. By an easy application of the pasting lemma for pullbacks, we then get that the square
$$
\bfig
\square<800,300>[f^\ast (E)`E`f^\ast (E)\times_Y f^\ast (E)`E\times_X E;`\Delta(f^\ast(p))`\Delta p`]
\efig
$$ 
is a pullback. Hence, $\Delta(f^\ast(p))$ is the pullback of the $L$-local map $\Delta p$, so it is itself $L$-local by \cref{def:refsub}~(1).\par
As for stability under dependent products, we get that $\Delta (\prod_f q)$ is $L$-local by applying Function Extensionality for dependent products (see \cite[Prop.~5.5 \& Rmk.~5.6]{l'loc}) to $q\in\E_{/Y}$ and to $\prod_f$, since $L$-local maps are closed under pullbacks and dependent products along arbitrary maps (by \cref{prop:propofreflsub} (ii)), and because $\Delta q$ is $L$-local by hypothesis. The last claim now follows, since $(-)^f\simeq\prod_f f^\ast (-)$.
\end{proof}

\begin{lemma}
\label{lm:weakcompproplsep}
Suppose given composable maps $X\stackrel{f}{\rightarrow}Y\stackrel{g}{\rightarrow}Z$ in $\E$ such that $g$ and $gf$ are $L$-separated. Then $f$ is $L$-separated. 
\end{lemma}
\begin{proof}
There are pullback squares
$$
\bfig
\hSquares(0,0)%
<300>[X\times_Y X`X`Y`X\times_Z X`X\times_Z Y`Y\times_Z Y;`f```\Delta g`\id_X\times_Z f`f\times_Z\id_Y]
\place(70,200)[\angle]
\place(1080,200)[\angle]
\efig
$$ 
in which all the vertical maps are $L$-local, since $\Delta g$ is $L$-local by hypothesis. If we let $p$ be the leftmost vertical map, we have that $p\circ\Delta f=\Delta(gf)$. We can then conclude using \cref{prop:weakcompproplocmap}.
\end{proof}
\begin{proposition}
\label{prop:lsepareloc}
The class $\M'$ of all $L$-separated maps is a local class of maps.
\end{proposition}
\begin{proof}
We already know that $\M'$ is pullback-stable. Suppose given a set of $L$-separated maps $f_i\colon X_i\rightarrow Y_i$, for $i\in I$. Let $p\colon P\rightarrow \coprod_i X_i$ be the pullback of $\coprod_i f_i$ with itself. Because colimits in $\E$ are universal, $P\simeq\coprod_{i} \iota_{X_i}^\ast (p)$, where $\iota_{X_i}$ is the coproduct inclusion of $X_i$. For a fixed $j\in I$, by definition of $p$ and of $\coprod_i f_i$, $\iota_{X_j}^\ast (p)$ is the pullback of $\coprod_i f_i$ along $\iota_{Y_j}\circ f_j$ and this pullback is just $X_j\times_{Y_j} X_j$:
$$
\bfig
\node xyx(0,0)[X_j\times_{Y_j} X_j]
\node x(600,0)[X_j]
\node cx(1200,0)[\coprod_i X_i]
\node x1(0,-300)[X_j]
\node y(600,-300)[Y_j]
\node cy(1200,-300)[\coprod_i Y_i]

\arrow[xyx`x;]
\arrow|a|[x`cx;\iota_{X_j}]
\arrow[xyx`x1;]
\arrow|l|[x`y;f_j]
\arrow|l|[cx`cy;\coprod_j f_j]
\arrow|b|[x1`y;f_j]
\arrow|b|[y`cy;\iota_{Y_j}]
\place(70,-90)[\angle]
\place(670,-90)[\angle]
\efig
$$ 
Here the right square is a pullback because coproducts in $\E$ are disjoint. Thus, $P=\coprod_i X_i\times_{Y_{i}}X_i$ and it follows that $\Delta(\coprod_i f_i)$ is the map $\coprod_i \Delta(f_i)$, which is 
$L$-local because the class of $L$-local maps is closed under coproducts (see \cref{lm:localmapsareclosedundercop}).\par
Finally, suppose given a pullback square
$$
\bfig
\square|mllm|/{>}`{>}`{>}`{->>}/<500,250>[E`M`X`Y;g`p`q`f]
\place(70,180)[\angle]
\efig
$$ 
where $f$ is an effective epi and $p$ is $L$-separated. We need to show that $q$ is $L$-separated. We have a commutative cube in $\E$
$$
\bfig
\cube|allb|/{->>}`{>}`{>}`{->>}/<800,470>[E\times_Y M`M\times_Y M`E`M;```]%
(300,-200)|amrb|/{>>}`{>}`{>}`-{>>}/<800,470>[E`M`X`Y;g`p`q`f]%
[```q]
\efig
$$
Here the bottom, front and right faces are all pullback squares. Since the composite of the back and right faces is also a pullback square, it follows that the back face is a pullback square, which implies all faces are pullback squares. Hence, $E\times_Y M\twoheadrightarrow M\times_Y M$ is an effective epimorphism and $E\times_X E=E\times_Y M$. Therefore, the diagonal $\Delta p$ can be identified with $\id_E\times_Y g$, which is then $L$-local (since $\Delta p$ is $L$-local by hypothesis). Thus, in the pullback square
$$
\bfig
\node e(0,0)[E]
\node m(700,0)[M]
\node eym(0,-300)[E\times_Y M]
\node mym(700,-300)[M\times_Y M]

\arrow|a|[e`m;g]
\arrow|l|[e`eym;\id_E\times_Y g]
\arrow|r|[m`mym;\Delta q]
\arrow|b|/->>/[eym`mym;]

\place(70,-70)[\angle]
\efig
$$ 
the left vertical map is $L$-local and the bottom horizontal map is an effective epimorphism. By \cref{lm:pllbckalongepi}, it follows that $\Delta q $ is also $L$-local, as required.
\end{proof}

Since every $L$-local map is $L$-separated, using \cref{prop:locclassareclass}, \cref{prop:localclassmapsareuniv} and \cite[Cor.~3.10]{univinloccarclosed} we get the following result.
\begin{corollary}
Let $\kappa$ be a regular cardinal such that the class of relatively $\kappa$-compact $L$-separated maps is classified by a univalent map ${u^L_\kappa}'\colon\ti{\U_\kappa^{L'}}\rightarrow~\U_\kappa^{L'}$. If $u^L_\kappa\colon\ti{\U_\kappa^{L}}\rightarrow \U_\kappa^{L}$ is the classifying map for relatively $\kappa$-compact $L$-local maps, then $u^L_\kappa$ is the pullback of ${u^L_\kappa}'$ along a monomorphism $\U_\kappa^{L}\rightarrowtail \U_\kappa^{L'}$.
\end{corollary}
By \cref{prop:lsepareloc}, $L$-separated maps satisfy a sufficient condition for being the class of local maps for a reflective subfibration on $\E$. The main goal of the companion paper \cite{l'loc} is to prove that this is indeed the case.
\begin{theorem}[{\cite[Thm.~4.3 \& Cor.~4.4]{l'loc}}]
Given any reflective subfibration $L_\bullet$ of an $\infty$-topos $\E$, there exists a reflective subfibration $L'_\bullet$ of $\E$ such that the $L'$-local maps are exactly the $L$-separated maps. Furthermore, if $L_\bullet$ is a modality, then so is $L'_\bullet$.
\end{theorem}
In particular, this implies that, for every $p\in\E_{/Z}$, there exists an $L'$\emph{-localization map}, $\eta'\colon p\rightarrow L'_Z (p)$, that is a map in $\E_{/Z}$ into an $L$-separated map which is initial among all maps from $p$ into an $L$-separated map. We have the following characterization result of $L'$-localization maps.
\begin{theorem}[{\cite[Thm.~3.10]{l'loc}}]
\label{thm:charoflprimeloc}
The following are equivalent for a map in $\E_{/Z}$
$$
\bfig
\Vtriangle|alr|<400,300>[X`X'`Z;\eta'`p`p']
\efig
$$
\begin{enumerate}
\item $\eta'$ is an $L'$-localization map of $p$.
\item $\eta'$ is an effective epimorphism and
$$
\bfig
\Vtriangle|alr|<400,300>[X`X\times_{X'}X`X\times_Z X;\Delta\eta'`\Delta p`]
\efig
$$
is an $L$-localization map of $\Delta p$.
\end{enumerate}
\end{theorem}

\section{Consequences of the existence of $L'_\bullet$}
\label{sec:conseqofexistofl'}
We explore here a few interactions between $L'_\bullet$ and $L_\bullet$, and we discuss those reflective subfibrations for which $L_\bullet=L'_\bullet$.\par
In \cref{sec:furtherinterbetweenlandl'} we study a few consequences of the existence of $L'$-localizations. In particular, \cref{prop:l'forstablefactsyst} constructs new stable factorization systems from a given one, and shows how the theory of reflective subfibrations can be used to prove theorems that make no reference to it. We then prove that $L'_1$ is almost left exact (\cref{prop:spantol'locspanislequiv}), paralleling the equivalent statement in \cite[\S 2.4]{locinhott}.\par In \cref{sec:selfsepreflsubf}, we introduce \emph{self-separated} reflective subfibrations as those $L_\bullet$ for which $L$-separated maps are $L$-local. We show that every self-separated reflective subfibration is associated to a \emph{(quasi-)cotopological localization} of $\E$ (\cref{thm:l=l'char}). The content of this section does not appear in \cite{locinhott}.

\subsection{Further interactions between $L_\bullet$ and $L'_\bullet$}
\label{sec:furtherinterbetweenlandl'}
Let $L_{\bullet}$ be a reflective subfibration on $\E$. We gather here some results relating $L_\bullet$ and $L'_\bullet$. 

\smallskip
We begin by looking at $L'$-connected maps, which are linked to $L$-connected maps in the expected way.
\begin{proposition}
\label{prop:charofl'connmaps}
A map $p\colon E\rightarrow X$ is $L'$-connected if and only if it is an effective epimorphism and $\Delta p\colon E\rightarrow E\times_X E$ is $L$-connected.
\end{proposition} 
\begin{proof}
By \cref{def:lconnmaps}, $p$ is $L'$-connected if and only if $p\colon p\rightarrow~\id_X$ is the $L'$-localization map of $p\in\E_{/X}$. The claim now follows from \cref{thm:charoflprimeloc}. 
\end{proof}

We use this result to construct stable factorization systems from a given one. 
\begin{proposition}
\label{prop:l'forstablefactsyst}
Let $\ff=(\Ll,\Rr)$ be a stable factorization system on an $\infty$-topos $\E$. Let $\Ll'$ be the class of maps $f$ in $\E$ which are effective epimorphisms and such that $\Delta f\in\Ll$, and let $\Rr'$ be the class of maps $g$ in $\E$ such that $\Delta g\in\Rr$. Then $\ff'=(\Ll',\Rr')$ is a stable factorization system on $\E$.
\end{proposition}
\begin{proof}
By \cref{thm:stablefactsystandmod}, there is a modality $L_\bullet=L^\ff_\bullet$ on $\E$ associated to $\ff$ and the modality $L'_\bullet$ of $L$-separated maps gives rise to a stable factorization system $\ff'=\ff_{L'}$. Since the $L^\ff_\bullet$-connected maps are the maps in $\Ll$, we conclude by \cref{prop:charofl'connmaps}.
\end{proof}

Recall from \cref{prop:existofsloc} the reflective subfibration $L^S_\bullet$ associated to a set $S$ of maps in $\E$. For these kinds of reflective subfibrations, the description of $L'_\bullet$ is particularly simple.

\begin{proposition}
\label{prop:l'forfloc}
Let $S$ be a set of maps in $\E$. Then $(L^S)'=L^{\Sigma S}$, where $\Sigma S$ is the set of maps given by the suspensions of the maps in $\E$.
\end{proposition} 
\begin{proof}
We show that the $L^S$-separated maps are the $L^{\Sigma S}$-local maps. We prove this for objects, the proof for maps being essentially the same, upon replacing $\E$ with a slice $\infty$-topos $\E_{/Z}$. We can reduce to the case where $S$ consists of a single map $f\colon A\rightarrow B$. Let $X\in\E$. We want to show that $X^{\Sigma f}\colon X^{\Sigma B}\rightarrow X^{\Sigma A}$ is an equivalence if and only if $(\Delta X)^{f\times X^2}\colon\Delta X^{\pr_B}\rightarrow\Delta X^{\pr_A} $ is an equivalence, where $f\times X^2\colon \pr_A\rightarrow \pr_B$ is a map in $\E_{/X^2}$, for $\pr_A\colon A\times X^2\rightarrow X^2$ the projection map, and similarly for $\pr_B$. Note that, since $\Sigma A$ comes with basepoints $S,N\colon 1\rightarrow \Sigma A$, $X^{\Sigma A}$ is an object over $X^2$. Similarly, $X^{\Sigma f}$ is a map over $X^2$ and it is an equivalence as such if and only if it is an equivalence as a map in $\E$, since the forgetful functor $\E_{/X}\rightarrow \E$ is conservative. We now show that $X^{\Sigma f}$ is $(\Delta X)^{f\times X}$.\par
Let $[A\times X^2,X^2]$ be the domain of $(\Delta X)^{\pr_A}$, when seen as a map in $\E$. By \cite[Lemma 2.5.5]{goodwillie}, there is a pullback square
$$
\bfig
\node axx(0,0)[{[}A\times X^2,X^2{]}]
\node xa(600,0)[X^A]
\node xx(0,-250)[X^2]
\node xxa(600,-250)[(X^2)^A]

\arrow[axx`xa;]
\arrow|l|[axx`xx;(\Delta X)^{\pr_A}]
\arrow|r|[xa`xxa;(\Delta X)^A]
\arrow|b|[xx`xxa;]

\place(70,-90)[\angle]
\efig
$$ 
where the bottom map is induced by $A\rightarrow 1$. If we paste this square with 
$$
\bfig
\node xa(0,0)[X^A]
\node xa1(600,0)[X^A]
\node xxa(0,-250)[(X^2)^A]
\node xaxa(600,-250)[(X^A)^2]

\arrow|a|[xa`xa1;\id]
\arrow|l|[xa`xxa;(\Delta X)^A]
\arrow|r|[xa1`xaxa;\Delta(X^A)]
\arrow|b|[xxa`xaxa;\simeq]

\place(70,-90)[\angle]
\efig
$$
we get that $(\Delta X)^{\pr_A}$ is also the pullback of $\Delta(X^A)$ along $c^2\colon X^2\rightarrow (X^A)^2$, where $c\colon X\rightarrow X^A$ is induced by $A\rightarrow 1$. But $(c^2)^\ast(\Delta (X^A))=(X\times_{X^A}X\rightarrow X^2)$ and $X\times_{X^A}X\simeq X^{\Sigma A}$, because $\Sigma A=1\coprod_A 1$. Therefore, $(\Delta X)^{\pr_A}=(X^{\Sigma A}\rightarrow X^2)$. Similarly, $(\Delta X)^{\pr_B}=(X^{\Sigma B}\rightarrow X^2)$ and $(\Delta X)^{f\times X}$ is $X^{\Sigma f}$.
\end{proof}

Let $L^0_\bullet$ be the $0$-truncated modality, as in \cref{ex:ntruncfactsyst}.  Consider the circle $S^1$ in the $\infty$-topos $\E=\infty\Gpd$ and fix a point in it. Then $\Omega S^1$ is $0$-truncated. (Recall that, for $X\in\E$ and $x\colon 1\rightarrow X$ a global element of $X$, $\Omega (X,x):= 1\times_X 1$.)  Therefore, $L^0 (\Omega S^1)\simeq\Omega S^1$. On the other hand, $L^0(S^1)$  is a point, because $S^1$ is $0$-connected. This simple observation shows that, in general, $L$-localization does not commute with taking loop objects. $L'$-localization can be used to fix this misbehaviour.
\begin{corollary}
\label{cor:locandloop}
Let $L_\bullet$ be a reflective subfibration on $\E$. Let $X\in\E$ be an object of $\E$ with a global element $x\colon 1\rightarrow X$. Set $\Omega X:=\Omega(X,x)$. Then $L(\Omega X)\simeq\Omega (L'X)$, where the loop object of $L'X$ is taken with respect to the global element $1\xrightarrow{x}X\xrightarrow{\eta'(X)}~L'X$. 
\end{corollary}
\begin{proof}
Since $\Omega X$ is the pullback along $(x,x)\colon 1\rightarrow X^2$ of $\Delta X$, the localization $L(\Omega X)$ is the pullback along $(x,x)$ of the $L$-localization of $\Delta X$ in $\E_{/X^2}$. By \cref{thm:charoflprimeloc}, the $L$-localization of $\Delta X$ is the map $X\times_{L'X}X\rightarrow X^2$, which can be obtained as the pullback of $\Delta (L'X)$ along $(\eta'(X))^2$. The claim then follows. 
\end{proof}

The following result relates the pullback of a cospan of objects in $\E$ to the pullback of the $L'$-localized span, and will be used in the next section.

\begin{proposition}[{\cite[Prop.~ 2.28]{locinhott}}]
\label{prop:spantol'locspanislequiv}
Let $Y\xrightarrow{g} X\xleftarrow{f} Z$ be maps in $\E$ and let $L'Y\xrightarrow{L'g} L'X\xleftarrow{L'f} L'Z$ be the associated cospan of $L'$-local objects. Then, the natural map $\psi\colon P\rightarrow Q$ induced on pullbacks is an $L_1$-equivalence.
\end{proposition}
\begin{proof}
The situation can be described by the diagram
$$
\bfig
\node yx(0,0)[Y\times X]
\node zz(800,0)[Z^2]
\node l'z2(1100,-300)[L'Z^2]
\node p(0,600)[P]
\node q(300,400)[Q]
\node z(800,600)[Z]
\node l'z(1100,400)[L'Z]
\node l'yl'x(300,-300)[L'Y\times L'X]

\arrow[q`l'z;]
\arrow||/@{->}_<>(0.4){q}/[q`l'yl'x;]
\arrow[p`z;]
\arrow|r|[z`l'z;\eta'(Z)]
\arrow|b|/{ -->}/[p`q;\psi]
\arrow|m|[zz`l'z2;\eta'(Z)^2]
\arrow|r|[l'z`l'z2;\Delta(L'Z)]
\arrow||/@{->}_<>(0.2){\eta'(Y)\times\eta'(X)}/[yx`l'yl'x;]
\arrow|b|[l'yl'x`l'z2;L'g\times L'f]
\arrow||/@{->}|!{(300,400);(300,-300)}\hole|(0.6){g\times f}/[yx`zz;]
\arrow|l|[p`yx;p]
\arrow||/@{->}|!{(300,400);(1100,400)}\hole|(0.6){\Delta Z}/[z`zz;]
\efig
$$
where the front and back squares are pullbacks. If we let $\eta \colon\Delta Z\rightarrow r_Z$ be the $L$-reflection map of $\Delta Z$ into $\D_{Z^2}$, we can expand the back and the right faces above as in the following diagram
$$
\bfig
\node p(0,0)[P]
\node z(900,0)[Z]
\node gfr(-500,-300)[(g\times f)^\ast R_Z]
\node r(400,-300)[R_Z]
\node zlzz(1400,-300)[Z\times_{L'Z}Z]
\node lz(2200,-300)[L'Z]
\node yx(0,-600)[Y\times X]
\node zz(900,-600)[Z^2]
\node lzlz(1700,-600)[L'Z^2]

\arrow[p`z;]
\arrow|m|[gfr`r;]
\arrow||/@{->}|!{(-500,-300);(400,-300)}\hole_<>(.3)p/[p`yx;]
\arrow|b|[yx`zz;g\times f]
\arrow/{@{->}@/^20pt/}/[z`lz;\eta'(Z)]
\arrow|l|[p`gfr;\bar{\eta}]
\arrow|m|[gfr`yx;(g\times f)^\ast r_Z]

\arrow|l|[z`r;\eta]
\arrow|r|[z`zlzz;\Delta(\eta'(Z))]
\arrow|l|[r`zz;r_Z]
\arrow[zlzz`zz;]
\arrow||/@{->}_<>(0.2){\Delta Z}/[z`zz;]
\arrow||/@{-->}|!{(900,0);(900,-600)}\hole^<>(.8){\exists !\phi}_<>(.8)\simeq/[r`zlzz;]
\arrow[zlzz`lz;]
\arrow|r|[lz`lzlz;\Delta(L'Z)]
\arrow|b|[zz`lzlz;\eta'(Z)^2]
\efig
$$
where the equivalence $\phi$ is given by \cref{thm:charoflprimeloc} (applied on the base $\infty$-topos $\E$). Note that $\li{\eta}$, being the reflection map of $p$ into $\D_{Y\times X}$, is an $(L_{Y\times X})$-equivalence. The bottom half of the diagram above gives a composite pullback square. So we see that $(g\times f)^\ast(r_z)$ is the pullback of $\Delta(L'Z)$ along $(\eta'(Z)^2)(g\times f)=(L'g\times L'f)(\eta'(Y)\times\eta'(X)).$ Therefore, the composite pullback square factors through $q$ as
$$
\bfig
\node gfr(0,0)[(g\times f)^\ast (R_Z)]
\node q(900,0)[Q]
\node lz(1800,0)[L'Z]
\node yx(0,-350)[Y\times X]
\node lylx(900,-350)[L'Y\times L'X]
\node lzlz(1800,-350)[L'Z^2]

\arrow[gfr`q;t]
\arrow[q`lz;]
\arrow|l|[gfr`yx;(g\times f)^\ast(r_Z)]
\arrow|m|[q`lylx;q]
\arrow|r|[lz`lzlz;\Delta(L'Z)]
\arrow|b|[yx`lylx;\eta'(Y)\times\eta'(X)]
\arrow|b|[lylx`lzlz;L'g\times L'f]
\efig
$$
since the right square is a pullback by definition. Therefore, the left square is also a pullback and the map $t$ is $L$-connected, because it is a pullback of the $L$-connected map $\eta'(Y)\times\eta'(X)$ (see \cite[Lemma 3.4]{l'loc}). It follows from these considerations that the map $\psi\colon P\rightarrow Q$ is given by the composite $t\li{\eta}$, where both $t$ and $\li{\eta}$ are $L_1$-equivalences (see also  \cref{lm:lequivandpullbdepsum} and \cref{def:lconnmaps}). Hence, $\psi$ is also an $L_1$-equivalence.
\end{proof}

\subsection{Self-separated reflective subfibrations}
\label{sec:selfsepreflsubf}
We study here the reflective subfibrations $L_\bullet$ on $\E$ for which $L_\bullet =L'_\bullet$ and show that they correspond to some left exact reflective subcategories of $\E$, the \emph{quasi-cotopological localizations} of $\E$. 
\begin{definition}
A reflective subfibration $L_\bullet$ on an $\infty$-topos $\E$ is \emph{self-separated} if every $L$-separated map is $L$-local.
\end{definition}
The existence of self-separated reflective subfibrations on $\E$ is related to a property of an $\infty$-topos called \emph{hypercompleteness}, which is discussed in detail in \cite[\S 6.5.2]{htt}. We gather here the main aspects of hypercompleteness that we need.

\begin{definition}
\label{def:hypercomplete}
Let $\E$ be an $\infty$-topos.
\begin{enumerate}
\item A map $f$ in $\E$ is called $\infty$\emph{-connected} if it is $n$-connected, for every $n\geq (-2)$. In particular, equivalences are $\infty$-connected.
\item An object $X$ in $\E$ is called \emph{hypercomplete} if, for every $\infty$-connected map $f\colon A\rightarrow B$, $X^f\colon X^B\rightarrow X^A$ is an equivalence in $\E$ or, equivalently,\break $\E(f,X)\colon \E(B,X)\rightarrow\E(A,X)$ is an equivalence in $\infty\Gpd$. In particular, $n$-truncated objects are hypercomplete. A map $p\colon E\rightarrow X$ is \emph{hypercomplete} if it is a hypercomplete object of $\E_{/X}$.

\item $\E$ is a \emph{hypercomplete} $\infty$-topos if every object in $\E$ is hypercomplete. Equivalently, $\E$ is hypercomplete if every $\infty$-connected map in $\E$ is an equivalence.  
\end{enumerate}
\end{definition}
Not every $\infty$-topos is hypercomplete. However, we have the following result.

\begin{proposition}
\label{prop:factsabouthypercompletion}
Let $\E$ be an $\infty$-topos and let $\E^{\wedge}$ be the full subcategory of $\E$ spanned by the hypercomplete objects. 
\begin{enumerate}
\item $\E^{\wedge}$ is an accessible, left exact, reflective subcategory of $\E$. In particular, it is an $\infty$-topos. As such, it is hypercomplete.
\item A map in $\E$ is hypercomplete if and only if it is right orthogonal to every $\infty$-connected map. For $X\in\E$, a map $\alpha$ in the $\infty$-topos $\E_{/X}$ is $\infty$-connected if and only if $\Sigma_X(\alpha)$ is $\infty$-connected in $\E$.  
\item There exists a modality $L^{\wedge}_\bullet$ on $\E$ for which the $L^{\wedge}$-equivalences are the $\infty$-connected maps, and the $L^{\wedge}$-local maps are the hypercomplete maps. We call this modality the \emph{hypercompletion modality} on $\E$. 
\end{enumerate}
\end{proposition}
\begin{proof}
The first part follows from \cite[Prop.~6.5.2.8]{htt} (see the discussion right after it) and from \cite[Lemma 6.5.2.12]{htt}. The second part is \cite[Rmk.~6.5.2.21]{htt}. For the last part, we can apply the results of \cref{prop:leftexactreflsubmod} to the left adjoint $\E\rightarrow\E^{\wedge}$ of the inclusion $\E^{\wedge}\subseteq ~\E$. In this way, we obtain a modality $L^{\wedge}_\bullet$ on $\E$ with the desired $L^{\wedge}$-equivalences and $L^{\wedge}$-local maps.
\end{proof}
We are now ready to study self-separated reflective subfibrations.
\begin{lemma}
\label{lemma:selfseplocarehyper}
Let $L_\bullet$ be a self-separated reflective subfibration on $\E$. Then every $L$-equivalence is an $\infty$-connected map and every hypercomplete map is $L$-local. In particular, if $\E$ is hypercomplete, $L_\bullet$ is the trivial reflective subfibration for which the $L$-equivalences are exactly the equivalences in $\E$, and every map is $L$-local.
\end{lemma}
\begin{proof}
Since $L_\bullet$ is self-separated, for every map $p$ in $\E$, if $\Delta p$ is $L$-local, then $p$ is $L$-local. Since equivalences are $L$-local, this implies that every monomorphism is $L$-local. It follows that every $n$-truncated map is $L$-local, due to the recursive characterization of $n$-truncated maps in terms of their diagonals (see \cite[Lemma~5.5.6.15]{htt}). Since, for every $X\in\E$, $L_X$-equivalences are left orthogonal to every map in $\D_{X}$ (see \cref{not:stuffaboutreflsub}), we get that every $L$-equivalence is $\infty$-connected. A hypercomplete map $p\in\E_{/X}$ is local with respect to \emph{all} $\infty$-connected map in $\E_{/X}$, so every hypercomplete map is $L$-local.
\end{proof}
When $\E$ is not hypercomplete, we can find non-trivial examples of self-separated reflective subfibrations.

\begin{definition}
\label{def:cotoploc}
Suppose $i\colon\D\hookrightarrow \E$ is a reflective subcategory of $\E$, with reflector $a\colon \E\rightarrow \D$. We say that $L:=ia$ is a \emph{quasi-cotopological localization} of $\E$ if it is left exact and, for every map $f$ in $\E$, if $Lf$ is an equivalence, then $f$ is $\infty$-connected.
\end{definition}
Note that the hypercompletion $L^{\wedge}\colon\E\rightarrow \E^{\wedge}$ is a quasi-cotopological localization.
\begin{remark}
In \cite[Def.~6.5.2.17]{htt}, Lurie calls a localization \emph{cotopological} if it is quasi-cotopological and accessible. 
\end{remark}
\begin{proposition}
\label{prop:cotlocareselfsep}
Let $L_\bullet$ be the modality associated to a quasi-cotopological localization $L\colon \E\rightarrow \E$ (see \cref{prop:leftexactreflsubmod}). Then $L_\bullet$ is self-separated.
\end{proposition}
\begin{proof}
By the construction of $L_\bullet$ given in \cref{prop:leftexactreflsubmod}, a map in $\E_{/Z}$ is an $L$-equivalence if and only if it is an $(L_1=L)$-equivalence in $\E$. Since $L$ is a quasi-cotopological localization, it follows that, for any $Z\in\E$ and any $p\in\E_{/Z}$, all reflection maps $\eta_Z(p)\colon p\rightarrow L_Z(p)$ are $\infty$-connected. In particular, they are effective epi. We now show that every $L$-separated object is $L$-local, the proof for maps being the same, but done in an appropriate slice category. Suppose that $X\in\E$ is such that $\Delta X$ is $L$-local and let $\eta\colon X\rightarrow LX$ be the reflection map of $X$. Using the definition of $L$-local maps from \cref{prop:leftexactreflsubmod}, since $L(X^2)\simeq (LX)^2$, the hypothesis that $\Delta X$ is an $L$-local map means that there is a pullback square
$$
\bfig
\node x(0,0)[X]
\node lx(500,0)[LX]
\node xx(0,-300)[X^2]
\node lxlx(500,-300)[LX^2]

\arrow[x`lx;\eta]
\arrow[x`xx;\Delta X]
\arrow|r|[lx`lxlx;\Delta (LX)]
\arrow|b|[xx`lxlx;\eta^2]

\place(70,-70)[\angle]
\efig
$$
Hence, the diagonal of $\eta$ is an equivalence, i.e., $\eta$ is a monomorphism. Since $\eta$ is also an effective epimorphism, it is an equivalence and, then, $X$ is $L$-local, as needed.
\end{proof}
It turns out that the quasi-cotopological localizations completely characterize self-separated reflective subfibrations on $\E$.
\begin{theorem}
\label{thm:l=l'char}
The following are equivalent, for a reflective subfibration $L_\bullet$ on $\E$.
\begin{enumerate}
\item $L_\bullet$ is self-separated.
\item $L_\bullet$ is the modality associated to a quasi-cotopological localization of $\E$.
\end{enumerate}
In this case, hypercomplete maps are $L$-local.
\end{theorem}
\begin{proof}
\cref{prop:cotlocareselfsep} is the statement that $(2)\implies (1)$. Thus, we need to show that $(1)\implies (2)$. For any reflective subfibration $L_\bullet$, given composable maps $f$ and $g$, if $f$ is $L$-separated and $g$ is $L$-local, then $f\circ g$ is $L$-separated, by \cite[Lemma 3.1]{l'loc}. Thus, if every $L$-separated map is $L$-local, $L_\bullet$ is a modality. By \cref{prop:spantol'locspanislequiv}, given maps $X\rightarrow Z\leftarrow Y$ in $\E$, the canonical map $X\times_Z Y \rightarrow LX\times_{LZ} LY$ is an $L$-equivalence and $LX\times_{LZ} LY$ is an $L$-local object, because $L_\bullet$ is a modality. This means that $ LX\times_{LZ} LY\simeq L(X\times_Z Y)$, so $L=L_1\colon\E\rightarrow\D\hookrightarrow \E$ is left exact. On the other hand,  if $L_\bullet$ is self-separated, \cref{lemma:selfseplocarehyper} says that every $L$-equivalence is $\infty$-connected. Therefore, $L\colon\E\rightarrow\E$ is a quasi-cotopological localization of $\E$. To conclude, we show that a map $f\colon X\rightarrow Y$ is $L$-local if and only if there is a pullback square
\begin{equation}
\label{eq:auxsqforl=l'}
\tag{$\bigtriangleup$}
\bfig
\square<400,250>[X`LX`Y`LY;\eta(X)`f`L(f)`\eta (Y)]
\efig
\end{equation}
If this is the case, then the local maps of $L_\bullet$ are the same as the local maps of the reflective subfibration associated to $L_1$, by \cref{prop:leftexactreflsubmod}, and therefore the two reflective subfibrations are the same. Suppose that $f\colon X\rightarrow Y$ is $L$-local. By \cite[Prop.~3.6]{l'loc}, there is a pullback square
$$
\bfig
\square<800,350>[X`L_{LY}(X)`Y`LY;\eta_{LY}(\eta(Y)f)`f`L_{LY}(\eta(Y)f)`\eta (Y)]
\place(70,230)[\angle]
\efig
$$
Since $Lf\in\D_{LY}$, there is a unique $\phi\colon L_{LY}(\eta(Y)f)\rightarrow Lf$ with $\phi\eta_{LY}(\eta(Y)f)=\eta(X)$ and $(Lf)\phi=L_{LY}(\eta(Y)f)$. Since $\eta_{LY}(\eta(Y)f)$ and $\eta(X)$ are $L_1$-e\-quiv\-a\-lences, so is $\phi$. Since $L_\bullet$ is a modality, $LX$ and $L_{LY}(X)$ are $L$-local objects, and then $\phi\colon L_{LY}(X)\rightarrow LX$ is an equivalence. It follows that \eqref{eq:auxsqforl=l'} is a pullback square.
\end{proof}

\bibliographystyle{amsalpha2} 
\bibliography{Locinaninftytopos.bib}

\providecommand{\bysame}{\leavevmode\hbox to3em{\hrulefill}\thinspace}
\providecommand{\MR}{\relax\ifhmode\unskip\space\fi MR }
\providecommand{\MRhref}[2]{%
  \href{http://www.ams.org/mathscinet-getitem?mr=#1}{#2}
}
\providecommand{\href}[2]{#2}
\begin{thebibliography}{ABFJ17b}

\bibitem[ABFJ17a]{agenBM}
M.~{Anel}, G.~{Biedermann}, E.~{Finster}, and A.~{Joyal}, \emph{{A Generalized
  Blakers-Massey Theorem}}, ArXiv e-prints (2017), \href
  {http://arxiv.org/abs/1703.09050} {\path{arXiv:1703.09050}}.

\bibitem[ABFJ17b]{goodwillie}
\bysame, \emph{Goodwillie's calculus of functors and higher topos theory},
  ArXiv e-prints (2017), \href {http://arxiv.org/abs/1703.09632}
  {\path{arXiv:1703.09632}}.

\bibitem[CORS18]{locinhott}
J.~D. {Christensen}, M.~{Opie}, E.~{Rijke}, and L.~{Scoccola},
  \emph{{Localization in Homotopy Type Theory}}, to appear in \textit{Higher
  Structures}, ArXiv e-prints (2018), \href {http://arxiv.org/abs/1807.04155}
  {\path{arXiv:1807.04155}}.

\bibitem[GK17]{univinloccarclosed}
D.~{Gepner} and J.~{Kock}, \emph{Univalence in locally cartesian closed
  {$\infty$}-categories}, Forum Math. \textbf{29} (2017), no.~3, 617--652,
  \href {http://dx.doi.org/10.1515/forum-2015-0228}
  {\path{doi:10.1515/forum-2015-0228}}.

\bibitem[{Kap}14]{joyalconjhott}
K.~R. {Kapulkin}, \emph{Joyal's {C}onjecture in homotopy type theory}, ProQuest
  LLC, Ann Arbor, MI, 2014, Thesis (Ph.D.)--University of Pittsburgh.

\bibitem[{Lur}09]{htt}
J.~{Lurie}, \emph{Higher topos theory}, Annals of Mathematics Studies, vol.
  170, Princeton University Press, Princeton, NJ, 2009, \href
  {http://dx.doi.org/10.1515/9781400830558} {\path{doi:10.1515/9781400830558}}.

\bibitem[MP12]{moreconcise}
J.~P. {May} and K.~{Ponto}, \emph{More concise algebraic topology}, Chicago
  Lectures in Mathematics, University of Chicago Press, Chicago, IL, 2012,
  Localization, completion, and model categories.

\bibitem[{Rez}15]{rezkprookofbm}
C.~{Rezk}, \emph{Proof of the blakers-massey theorem}, 2015.

\bibitem[RSS17]{modinhott}
E.~{Rijke}, M.~{Shulman}, and B.~{Spitters}, \emph{{Modalities in homotopy type
  theory}}, ArXiv e-prints (2017), \href {http://arxiv.org/abs/1706.07526}
  {\path{arXiv:1706.07526}}.

\bibitem[{Shu}19]{intlang}
M.~{Shulman}, \emph{{All $(\infty,1)$-toposes have strict univalent
  universes}}, ArXiv e-prints (2019), \href
  {http://arxiv.org/abs/arXiv:1904.07004v1} {\path{arXiv:arXiv:1904.07004v1}}.

\bibitem[UF13]{hott}
The Univalent~Foundations Program, \emph{Homotopy type theory---univalent
  foundations of mathematics}, The Univalent Foundations Program, Princeton,
  NJ; Institute for Advanced Study (IAS), Princeton, NJ, 2013.

\bibitem[{Ver}]{l'loc}
M.~{Vergura}, \emph{{$L'$-localization in an $\infty$-topos}}, ArXiv preprint.

\end{thebibliography}

\end{document}